\documentclass[11pt]{amsart}
\usepackage{amsmath,amsthm,amssymb,enumerate,mathscinet,mathtools}
\usepackage{fullpage}
\usepackage{graphicx,xcolor,pgfplots}
\usepackage{verbatim}
\usepackage{mathrsfs}

\usepackage{xcolor}
\usepackage[backref=section]{hyperref}
\hypersetup{
	colorlinks=true,
	linkcolor={red!60!black},
	citecolor={green!60!black},
	urlcolor={blue!60!black},
}

\usepackage{enumitem}
\newcommand{\customlabel}[2]{#2\def\@currentlabel{#2}\label{#1}}

\pgfplotsset{compat=1.18} 
\newcommand{\EE}{\mathbb{E}}
\newcommand{\NN}{\mathbb{N}}
\newcommand{\PP}{\mathbb{P}}

\def\Prob{\mathbb{P}}

\def\F{\mathcal{F}}

\usepackage{scalerel}
\usepackage{dsfont}
\usepackage{blindtext}
\usepackage{multicol}
\newtheorem{theorem}{Theorem}[section]

\newtheorem{lemma}[theorem]{Lemma}
\newtheorem{question}[theorem]{Question}
\newtheorem{proposition}[theorem]{Proposition}
\newtheorem{corollary}[theorem]{Corollary}
\newtheorem{conjecture}[theorem]{Conjecture}
\newtheorem{claim}[theorem]{Claim}
\newtheorem{definition}[theorem]{Definition}

\newtheorem{thm}[theorem]{Theorem}
\newtheorem{prop}[theorem]{Proposition}
\newtheorem{example}[theorem]{Example}
\newtheorem{observation}[theorem]{Observation}

\newtheorem{conj}[theorem]{Conjecture}

\newtheorem{remark}[theorem]{Remark}

\newtheorem{ques}[theorem]{Question} 
\numberwithin{equation}{section}

\def\F{\mathcal{F}}

\def\eps{\varepsilon}
\newcommand{\ep}{\varepsilon}

\newcommand{\floor}[1]{\lfloor#1\rfloor}
\newcommand{\ceil}[1]{\lceil#1\rceil}

\newcommand{\clique}[1]{\overline{K}_{#1}}

\newcommand{\ck}[1]{\vec{C}_{#1}^k}

\newcommand{\fromab}[1]{\overrightarrow{\prescript{ab}{}{#1}}}
\newcommand{\toyz}[1]{\overrightarrow{#1^{yz}}}

\def\Prob{\mathbb{P}}

\makeatletter
\newcommand*{\rom}[1]{\expandafter\@slowromancap\romannumeral #1@}
\makeatother

\newenvironment{claimproof}[1][\underline{Proof}]{\begin{proof}[#1]}{\end{proof}}

\def\COMMENT#1{}
\let\COMMENT=\footnote

\newcommand{\LD}[1]{\textcolor{red}{[LD: #1]}}

\newcommand{\JH}[1]{\textcolor{cyan}{[JH: #1]}}

\usepackage{caption}

\usepackage{tikz}
\usetikzlibrary{arrows.meta, positioning}
\usetikzlibrary{patterns}

\allowdisplaybreaks

\title{Powers of Hamilton cycles in oriented and directed graphs}

\author{Louis DeBiasio \and Jie Han \and Allan Lo \and Theodore Molla \and Sim\'on Piga \and Andrew Treglown}

\thanks{
\\ \indent LD: Department of Mathematics, Miami University, Oxford, OH. Email: \texttt{debiasld@miamioh.edu}. Research supported in part by NSF grant DMS-1954170.
\\ \indent JH: School of Mathematics and Statistics and Center for Applied Mathematics, Beijing Institute of Technology, Beijing, China. Email: \texttt{han.jie@bit.edu.cn}. Research partially supported by National Natural Science Foundation of China (12371341).
\\ \indent AL: School of Mathematics, University of Birmingham, United Kingdom. Email: \texttt{s.a.lo@bham.ac.uk}. Research supported by EPSRC grants EP/V002279/1 and EP/V048287/1.
\\ \indent TM: Department of Mathematics and Statistics, University of South Florida, Tampa, FL. Email: \texttt{molla@usf.edu}.   Research supported in part by NSF grants DMS-1800761 and DMS-2154313.
\\ \indent SP: Fachbereich Mathematik, Universit\"at Hamburg, Hamburg, Germany. Email: \texttt{simon.piga@uni-hamburg.de}. Research partially supported by EPSRC grant EP/V002279/1.
\\ \indent AT: School of Mathematics, University of Birmingham, United Kingdom. Email: \texttt{a.c.treglown@bham.ac.uk}. Research supported by EPSRC grants EP/V002279/1 and EP/V048287/1.}

\begin{document}

\maketitle

\begin{abstract}
The P\'osa--Seymour conjecture determines the minimum degree threshold for forcing the $k$th power of a Hamilton cycle in a graph. After numerous partial results, Koml\'os,  S\'ark\"ozy and Szemer\'edi proved the conjecture for sufficiently large graphs. In this paper we focus on the analogous problem for digraphs and for oriented graphs. We asymptotically determine the minimum total degree threshold for forcing the square of a Hamilton cycle in a digraph. We also give a conjecture on the corresponding threshold for $k$th powers of a Hamilton cycle more generally. For oriented graphs, we provide a minimum semi-degree condition that forces the $k$th power of a Hamilton cycle;   although this minimum semi-degree condition  is not tight, it does provide the correct order of magnitude of the  threshold. Tur\'an-type problems for oriented graphs are also discussed.
\end{abstract}

\section{Introduction}
A widely studied generalization of the notion of a Hamilton cycle is that of the \emph{$k$th power of a Hamilton cycle}:
the $k$th power of a Hamilton cycle $C$ is obtained from $C$ by adding an edge
between every pair of vertices of distance at most $k$ on $C$. We usually call the $2$nd power of a Hamilton cycle the \emph{square of a Hamilton cycle}. As well as being natural objects in their own right, finding the $k$th power of a Hamilton cycle in a graph $G$ ensures that $G$ 
contains other well-studied graph structures. For example, an $n$-vertex square of a Hamilton cycle contains every possible collection of vertex-disjoint paths and cycles on $n$ vertices. Further, if $k+1$ divides $n$, then an $n$-vertex  $k$th power of a Hamilton cycle
contains a $K_{k+1}$-factor.\footnote{An \emph{$H$-factor} in a graph $G$ is a collection of vertex-disjoint copies of a graph $H$ in $G$ that together cover all the vertices in $G$.} Powers of Hamilton cycles have also been used as the `building blocks' for proving several \emph{bandwidth theorems}; see, e.g., \cite{bandwidth, eb, staden}.

A major branch of extremal graph theory concerns  minimum degree conditions that 
force a spanning structure in a graph. For example, Dirac's theorem asserts that 
every graph $G$ on $n \geq 3$ vertices and of minimum degree $\delta (G) \geq n/2$ 
contains a Hamilton cycle. The following famous conjecture provides a generalization
of Dirac's theorem to powers of Hamilton cycles.
\begin{conj}[P\'osa and Seymour, see~\cite{posa, seymour}]\label{seymour}
Let $G$ be a graph on $n\geq k\geq 2$ vertices.
If $\delta (G) \geq \frac{k}{k+1} n$
then $G$ contains the $k$th power of a Hamilton cycle.
\end{conj} 
Note that the minimum degree condition in Conjecture~\ref{seymour} cannot be lowered. Indeed, in the case when $k+1$ divides $n$, consider the complete
$(k+1)$-partite graph $G$ in which all but two classes have size $n/(k+1)$, one class has size $n/(k+1)+1$, and the last class has size $n/(k+1)-1$. Then $G$ does not contain a $K_{k+1}$-factor and $\delta(G)=kn/(k+1)-1$.

Whilst P\'osa's conjecture (the $k=2$ case) was posed in the early 1960s, and Seymour's conjecture (for arbitrary $k$) in 1974, it was not until the 1990s that significant progress was made on the problem. Indeed, after
 several partial results towards the P\'osa--Seymour conjecture  (see, e.g.,~\cite{fan0, fan1, fan2, fan3, fau, kssposa, kssjgt}),
Koml\'os, S\'ark\"ozy, and Szemer\'edi~\cite{kss} applied 
Szemer\'edi's regularity lemma to
prove Conjecture~\ref{seymour} for sufficiently large graphs. Subsequently,
proofs of P\'osa's conjecture for large graphs have been obtained that avoid the regularity lemma~\cite{cdk, lev}.

\subsection{Powers of Hamilton cycles in digraphs}
It is also natural to study powers of Hamilton cycles in {directed graphs}.
Recall that \emph{digraphs} are graphs such that every pair of vertices has at most two edges
between them, at most one oriented in each direction. \emph{Oriented graphs} are orientations of simple
graphs; so there is at most one directed edge between any pair of vertices. 
A \emph{tournament} is an oriented complete graph.
Note that
oriented graphs are a subclass of digraphs.
In this setting, the $k$th power of a Hamilton cycle $C$ 
is the digraph obtained from
$C$ by adding a directed edge from $x$ to $y$ if there is a directed path of length at most $k$ from $x$
to $y$ on $C$. 

Given a digraph $G$ and $x \in V(G)$, we write $d^+_G(x)$ (or simply $d^+(x)$) for the \emph{outdegree of $x$ in $G$} and  $d^-_G(x)$ (or simply $d^-(x)$) for the \emph{indegree of $x$ in $G$}.
The \emph{minimum semi-degree $\delta^0(G)$} of $G$ is the minimum of all the in- and outdegrees of the vertices in $G$. The \emph{minimum total degree $\delta(G)$} is the minimum number of edges incident to a vertex in $G$.
Ghouila-Houri~\cite{GhouilaHouri} proved that every strongly connected $n$-vertex digraph $G$ with minimum total degree $\delta(G)\geq n$ contains a  Hamilton cycle. Note that there are $n$-vertex digraphs $G$ with $\delta(G) = \floor{3n/2} - 2$ that are not strongly connected (and thus do not contain a  Hamilton cycle), so the strongly connected condition in Ghouila-Houri's theorem is necessary.\footnote{See Proposition~\ref{propextremal} below for a generalized version of this observation.} 
An immediate consequence of Ghouila-Houri's theorem is that having minimum semi-degree $\delta^0(G) \geq n/2$ forces a  Hamilton cycle, and this is best possible.

The problem of determining the minimum semi-degree threshold that forces the $k$th power of a Hamilton cycle in a digraph was raised in~\cite{treg}. Indeed, 
as stated in~\cite{treg}, one would expect a positive answer to the following question.
\begin{ques}\label{ques1}
    Does every $n$-vertex digraph $G$ with 
    $\delta ^0 (G) \geq \frac{k}{k+1}n$ contain the 
    $k$th power of a Hamilton cycle?
\end{ques}
 By replacing edges with `double edges' in the extremal example of the P\'osa--Seymour conjecture, one can see that, if true, the  minimum semi-degree condition in Question~\ref{ques1} would be tight.
 Question~\ref{ques1} does seem to be rather challenging, and we are unaware of any progress on the problem.

An aim of this paper is to raise the analogous question for the minimum total degree threshold; we  propose the following conjecture.
\begin{conj}\label{conjnew}
Let $k \in \mathbb{N}$ and suppose $n \in \mathbb N$ is sufficiently large. 
Write $n=(k+3)q+r$ where $q,r \in \mathbb Z$ and $0\leq r\leq k+2$. 
Every $n$-vertex digraph $G$ with  
\begin{align*}
\delta(G)\geq  \begin{cases} 
      2\ceil{(1-\frac{1}{k+3})n}-3 & \text{ if } r = k+2 ,\\
      2\ceil{(1-\frac{1}{k+3})n}-2 & \text{ if } r = k \text{ or } r=k+1 , \\
      2\ceil{(1-\frac{1}{k+3})n}-1 & \text{ otherwise,} \\      
   \end{cases}
\end{align*}
contains the $k$th power of a Hamilton cycle.
\end{conj}

In Section~\ref{sec:construction} we provide an extremal example that shows, if true, the minimum total degree condition in
Conjecture~\ref{conjnew} is  best possible (see Proposition~\ref{propextremal}).  Also note that Ghouila-Houri's theorem implies that Conjecture~\ref{conjnew} holds for $k=1$ since an $n$-vertex digraph $G$ with $\delta(G)\geq \floor{\frac{3n}{2}}-1$ is strongly connected.  
On this note, one may wonder if it is possible to significantly relax the minimum total degree condition in Conjecture~\ref{conjnew} at the expense of introducing some strong connectivity condition (perhaps of a similar form to the conclusion of the statement of Lemma~\ref{lem:manyCon}). We suspect 
that this is true and, moreover, further progress on Conjecture~\ref{conjnew} is likely to provide insight into precisely what form such a statement should take.

Our first result yields an asymptotic version of Conjecture~\ref{conjnew} in the case of the square of a Hamilton cycle (i.e.,~$k=2$).
\begin{theorem}\label{thm:square}
Given any $\eta>0$, there exists $n_0 \in \mathbb N$ so that for any $n \geq n_0$
the following holds. If $G$ is an $n$-vertex digraph  with 
$$
\delta (G) \geq \left (\frac{8}{5}+\eta \right ) n,
$$
then $G$ contains the square of a Hamilton cycle.
\end{theorem}
In Section~\ref{81} we explain why we believe it challenging to generalize Theorem~\ref{thm:square} to an asymptotic solution of Conjecture~\ref{conjnew} for all $k \in \mathbb N$. In particular, whilst some of our auxiliary results for Theorem~\ref{thm:square} apply to this more general question, the main stumbling block is establishing a suitable \textit{connecting lemma}.

\subsection{Powers of Hamilton cycles in oriented graphs}
There has also been interest in powers of Hamilton cycles in oriented graphs. In this setting, the emphasis has been on
the study of minimum semi-degree results rather than minimum total degree results; this is natural since if $G$ is an $n$-vertex transitive tournament then $\delta(G)=n-1$ and $G$ does not contain a (power of a) Hamilton cycle.

Answering a question of Thomassen from 1979, Keevash, K\"uhn, and Osthus~\cite{kkolms} proved that every sufficiently large $n$-vertex oriented graph with  $\delta ^0 (G) \geq (3n - 4)/8 $ contains a Hamilton cycle.
Moreover, the minimum semi-degree condition here cannot be lowered.

The second goal of this paper is to study the minimum semi-degree threshold for forcing the $k$th power of a Hamilton cycle in an oriented graph. 

When $G$ is a tournament then much is known. 
As discussed in~\cite{DMS}, every $n$-vertex tournament~$G$ with
$\delta^0 (G) \geq (n-2)/4$ contains a Hamilton cycle, and this degree condition is best possible. Bollob\'as and  H\"aggkvist~\cite{BH} proved that actually one
only needs to boost the minimum semi-degree  slightly to force the $k$th power of a Hamilton cycle: that is, for a fixed $k \in \mathbb N$, $\delta^0 (G) \geq (1+o(1))n/4$ ensures the $k$th power of a Hamilton cycle in an $n$-vertex  tournament $G$.
This result has recently been significantly refined by Dragani\'c,  Munh\'a Correia, and  Sudakov~\cite{DMS}, who proved that one can take 
$\delta^0 (G) \geq n/4+ cn^{1-1/\lceil k/2 \rceil}$ here, for some constant $c=c(k)$.

Dragani\'c,  Dross,  Fox, Gir{\~a}o, Havet, Kor{\'a}ndi, Lochet,  Correia,  Scott, and Sudakov~\cite{etal} proved that every tournament contains the $k$th power of path of length at least $\frac{n}{2^{4k+6}k}$ which is close to best possible since there are tournaments where the longest $k$th power of path has length less than $\frac{k(k+1)n}{2^k}$.  They also proved an exact result for square paths; that is, every tournament contains a square path of length at least $\ceil{2n/3}-1$, and this is best possible.  

For oriented graphs which are not tournaments, essentially nothing is known about $k$th powers of paths and cycles (whether it be short cycles or Hamilton cycles).  
Until now it has not even been proven that a minimum semi-degree of $\delta ^0 (G) \geq (1-\eps)n/2 $ suffices to force the square of a Hamilton cycle (for some tiny $\eps >0$).
The next theorem gives such a result. In fact, the result holds for $k$th powers of Hamilton cycles more generally, and actually determines the `order of magnitude' of the function between~$\eps$ and~$k$.

\begin{theorem}\label{thm:main}
For any $k\geq 2$ there 
exists an $n_0\in \mathbb N$ such that the following holds for all $n \ge n_0$.
Suppose $G$ is an $n$-vertex oriented graph with
$$\delta^0(G)\geq \left (\frac{1}{2}-\frac{1}{10^{6000k}} \right )n.$$ Then $G$ contains the $k$th power of a Hamilton cycle.
Furthermore, for every~$k\geq 15$ and sufficiently large~$n\in \mathbb N$, there is an~$n$-vertex oriented graph~$R_k$ with
    $\delta^0(R_k) > (\frac{1}{2}-\frac{4}{2^{k/5}})n$
    that does not contain the  $k$th power of a Hamilton cycle.
\end{theorem}
The furthermore part of Theorem~\ref{thm:main} is proven via Proposition~\ref{prop:growingk}
in Section~\ref{sec:ex}.
We suspect that it may be well out of reach to determine (even asymptotically) the minimum semi-degree threshold for forcing the $k$th power of a Hamilton cycle.
In fact, as indicated above, it has been a challenge to find the `right' candidate for an extremal example even for the $k=2$ case of the problem.
Treglown~\cite{treg} provided a construction that shows one requires a minimum semi-degree of at least $\delta ^0 (G) \geq 5n/12$. Later DeBiasio, cf. \cite[Section~1]{DMS}, used a slightly unbalanced blow-up of the Paley tournament on seven vertices to show that $\delta ^0 (G) \geq 3n/7 -1$ is necessary.
We give another example of an oriented graph with large minimum semi-degree and 
no square of a Hamilton cycle, beating all previous known constructions.

\begin{proposition} \label{prop:example}
   Given any~$n\in 11\mathbb N$, there is an~$n$-vertex oriented graph $G_2$ with $\delta (G_2)\geq 5n/11-2$ that does not contain the square of a Hamilton cycle.
\end{proposition}

The paper is organized as follows. In the next section we give an overview of the proofs of Theorems~\ref{thm:square} and \ref{thm:main}. In particular, each of these proofs rely on their own \emph{absorbing, connecting} and \emph{almost covering} lemmas.
In Sections~\ref{sec:p1} and~\ref{sec:p2} we prove these auxiliary lemmas for Theorems~\ref{thm:square} and~\ref{thm:main} respectively. 
Prior to this, in Section~\ref{sec:construction}
we provide the extremal example for Conjecture~\ref{conjnew} as well as the constructions  $R_k$ and $G_2$ from Theorem~\ref{thm:main} and Proposition~\ref{prop:example}.
In Section~\ref{sec:pre} we introduce some useful tools including Szemer\'edi's regularity lemma. 
The proofs of Theorems~\ref{thm:square} and~\ref{thm:main} are presented in Section~\ref{sec:6}.
Finally, we give some concluding remarks and results in Section~\ref{sec:conc}. In particular, in Section~\ref{sec:turan} we discuss the Tur\'an problem for oriented graphs. 

\subsection*{Notation}
Throughout, $\mathbb N$ denotes the set of positive integers (i.e., it does not contain $0$).

Let $G$ be a digraph. We define $|G|:=|V(G)|$ and $e(G):=|E(G)|$. Given $x \in V(G)$, we write $N^+_G(x)$ for the \emph{out-neighborhood of $x$ in $G$} and write $N^-_G(x)$ for the \emph{in-neighborhood of $x$ in $G$}.
Thus, $|N^+_G(x)|=d^+_G(x)$ and $|N^-_G(x)|=d^-_G(x)$. Given $Y\subseteq V(G)$ we define
$N^+_G(x,Y):= N^+_G(x) \cap Y$ and 
$N^-_G(x,Y):= N^-_G(x) \cap Y$.
Set $d^+_G(x,Y):=|N^+_G(x,Y)|$ and 
$d^-_G(x,Y):=|N^-_G(x,Y)|$, and let 
$d_G(x,Y):=d^+_G(x,Y)+ d^-_G(x,Y)$. We define $d_G(x):=d_G(x,V(G))$.

Given two vertices $x$ and $y$ of  $G$, we write $xy$ for the edge directed from $x$ to $y$.
Given subsets $A,B\subseteq V(G)$ (not-necessarily disjoint), let $E_G(A,B)$ (or simply $E(A,B)$) be the set of all $xy\in E(G)$ such that $x\in A$ and $y\in B$.  
Let $e_G(A,B):=|E(A,B)|$; we omit the subscript $G$ here when the digraph $G$ is clear from the context.  Note that $e_G(A,B)=\sum_{v\in A}d^+_G(v,B)$.

If $A, B\subseteq V(G)$ are disjoint then we define $G[A,B]$ to be the subdigraph of $G$ where $V(G[A,B])=A \cup B$ and $E(G[A\cup B])=E_G(A,B)$. 
Given $X \subseteq V(G)$, we write $G[X]$ for the subdigraph of $G$ induced by $X$. We write $G\setminus X$ for the  subdigraph of $G$ induced by $V(G)\setminus X$.

We write $C_k$ for the directed cycle on $k$ vertices.
Given a digraph $G$, the \emph{$k$th power of $G$} is the digraph obtained from $G$ as follows: for each distinct $x,y \in V(G)$, add the directed edge $xy$ if  there is a  directed path of length at most $k$  from $x$ to $y$ in $G$.
For brevity we call the $k$th power of a directed path a \emph{$k$-path} and the $k$th power of a directed cycle a \emph{$k$-cycle}. We write $C^k_{\ell}$ to denote the 
$k$-cycle on $\ell$ vertices. 

Given a (di)graph $G$ and $t\in\mathbb{N}$, we let $G(t)$ denote the~\emph{$t$-blow-up of~$G$}. 
More precisely, $V(G(t)):= \{v^{j} \colon v \in V(G) \text{ and } j\in [t]\}$ and $E(G(t)) := \{v^mw^{\ell} :  v w \in E(G) \text{ and }m,\ell\in [t]\}.$

We say that an oriented graph $G$ is \emph{semi-regular} if for all $v\in V(G)$, $|d^+_G(v)-d^-_G(v)|\leq 1$.

Given (di)graphs $G$ and $H$, an \emph{$H$-tiling in $G$} is a collection of vertex-disjoint copies of $H$ in $G$. An \emph{$H$-factor} in $G$ is a collection of vertex-disjoint copies of $H$ in $G$ that together cover $V(G)$.

Throughout the paper, we
omit all floor and ceiling signs whenever these are not crucial.
The constants in the hierarchies used to state our results are chosen from right to left.
For example, if we claim that a result holds whenever $0< a\ll b\ll c\le 1$, then 
there are non-decreasing functions $f:(0,1]\to (0,1]$ and $g:(0,1]\to (0,1]$ such that the result holds
for all $0<a,b,c\le 1$  with $b\le f(c)$ and $a\le g(b)$.  
Note that $a \ll b$ implies that we may assume in the proof that, e.g., $a < b$ or $a < b^2$.

\section{Overview of the proofs of Theorems~\ref{thm:square} and \ref{thm:main}}\label{sec:overview}

The proofs of both  Theorems~\ref{thm:square} and \ref{thm:main} are  similar and follow the same high-level strategy, though the details in each case are  different. 
In particular,  for both we use the \emph{connecting--absorbing method}, a technique first developed by R\"odl,  Ruci\'nski and  Szemer\'edi~\cite{RRS}. 
Suppose one wishes to embed the $k$th power of a Hamilton cycle in an $n$-vertex digraph $G$, and let $0<\eps \ll \eta \ll 1$.
Then, roughly speaking, an application of this method consists of three main steps:
\begin{itemize}
\item {\bf Step 1, the absorbing $k$-path $P_A$.} Find a $k$-path $P_A$ in $G$ such that $|P_A|\leq \eta n$. 
The $k$-path~$P_A$ has the property that given \emph{any} set $L \subseteq V(G) \setminus V(P_A)$ such that $|L|\leq 2 \eps n$, $G$ contains a $k$-path $P$ with vertex set $V(P_A) \cup L$, where the first $k$ vertices on $P$ are the same as the first $k$ vertices on $P_A$; similarly, the  last $k$ vertices on $P$ are the same as the last $k$  vertices on $P_A$.
\item {\bf Step 2, the reservoir set $\mathfrak R$.} Let $G':=G\setminus V(P_A)$. 
Find a  set $\mathfrak R \subseteq V(G')$ such that $|\mathfrak R|\leq \eps n$ and so that $\mathfrak R$ has the following property: given {arbitrary} disjoint \emph{ordered} $k$-sets $X,Y \subseteq V(G)$, there are many short  $k$-paths $P$ in $G$ so that  the first $k$ vertices on $P$ are the elements of $X$, ordered as in $X$;  the last $k$ vertices on $P$ are the elements of $Y$, ordered as in $Y$; and $V(P)\setminus (X \cup Y) \subseteq \mathfrak R$.
\item {\bf Step 3, almost covering with $k$-paths.} Let $G'':= G'\setminus \mathfrak R$.
Find a collection $\mathcal P$ of a bounded number of vertex-disjoint $k$-paths in $G''$ that together cover all but at most $\eps n$ of the vertices in $G''$.
\end{itemize}
These three steps then yield the $k$th power of a Hamilton cycle in  $G$. Indeed, one can use the reservoir set~$\mathfrak R$ to connect together all of the $k$-paths in $\mathcal P$ and $P_A$ into a single $k$-cycle $C^k$ in $G$ that covers all but at most $\eps n$ of the vertices from $G''$ and some of the vertices from $\mathfrak R$.
In total at most $2\eps n$ vertices in $G$ are not covered by $C^k$; these can then be absorbed by $P_A$ to obtain the $k$th power of a Hamilton cycle in  $G$.

As mentioned earlier, the proofs of
 Theorems~\ref{thm:square} and~\ref{thm:main} each rely on their own {absorbing, connecting} and {almost covering} lemmas.
 The almost covering lemmas are used to complete Step~3. Roughly speaking, the connecting lemmas ensure that for any disjoint ordered $k$-sets $X,Y \subseteq V(G)$ we can
 find many short $k$-paths $P$ in $G$ 
so that  the first $k$ vertices on $P$ are the elements of $X$, ordered as in $X$, and  the last $k$ vertices on $P$ are the elements of $Y$, ordered as in $Y$. These connecting lemmas are then used to construct the reservoir in Step~2.

In fact, the connecting lemmas are also used in Step~1. Indeed, the absorbing lemmas establish that for every vertex $v \in V(G)$, there are many short $k$-paths $P_v$ in $G$ with the property that one can insert $v$ into the middle of $P_v$ so that the resulting digraph is still a $k$-path.
By randomly sampling amongst all such $k$-paths $P_v$ (for all $v \in V(G)$), and then joining the selected $k$-paths up via the connecting lemma, one can obtain the absorbing $k$-path $P_A$ from Step~1.

In the case of Theorem~\ref{thm:main} things are a little more subtle than we have indicated above. Indeed, we cannot actually achieve Step~2 for \emph{arbitrary} ordered $k$-sets $X,Y \subseteq V(G)$ since there may be a choice of $X$ and $Y$ for which the vertices in $X$ do not even have a single common out-neighbor (or  
the vertices in $Y$ do not  have a  common in-neighbor).
Thus, we need to argue more carefully to ensure we only ever  connect between `well-behaved' $X,Y \subseteq V(G)$.

By applying two results from~\cite{cliquetilings}, the proofs of the absorbing and almost covering lemmas for Theorem~\ref{thm:square} are not too difficult. The main work for this theorem is proving the connecting lemma. 
The proof of Theorem~\ref{thm:main} is a little more involved. In the proofs of the absorbing and connecting lemmas we make use of the method of \emph{dependent random choice}.
The proof of the almost covering lemma is quite non-standard, and we apply the aforementioned result of Dragani\'c,  Munh\'a Correia, and  Sudakov~\cite{DMS} on powers of Hamilton cycles in tournaments of large minimum semi-degree.

\section{The extremal examples}\label{sec:construction}
\subsection{The extremal example for Conjecture~\ref{conjnew}}

The following provides an extremal construction $G$ for Conjecture~\ref{conjnew}.  

\begin{proposition}\label{propextremal}
Let $k,q\in \mathbb N$ and $r\in \mathbb Z$  such that  $n=(k+3)q+r$ where $0\leq r\leq k+2$.  There exists an $n$-vertex digraph $G$ with 
\[\delta(G)= \begin{cases} 
        2\ceil{(1-\frac{1}{k+3})n}-4 & \text{ if } r=k+2, \\
        2\ceil{(1-\frac{1}{k+3})n}-3 & \text{ if } r = k \text{ or } r=k+1 , \\
        2\ceil{(1-\frac{1}{k+3})n}-2 & \text{ otherwise,}
   \end{cases}
\]
that does not contain the $k$th power of a Hamilton cycle. 
\end{proposition}

\begin{proof}
Let $k,q\in \mathbb N$ and $r\in \mathbb Z$  such that  $n=(k+3)q+r$ where $0\leq r\leq k+2$. Define integers $r_1,\dots, r_{k+1}$ as equally as possible so that $2 \geq r_1 \geq r_2 \geq \dots \geq r_{k+1} \geq 0$ 
and $r=\sum_{i=1}^{k+1}r_i$.

Let $G$ be the $n$-vertex digraph consisting  of $k-1$ independent sets $V_1,\dots, V_{k-1}$  and two other classes $V_k$ and $V_{k+1}$  so that there are
 all possible double edges going out of the $k-1$ independent sets;
  all possible double edges inside of $V_{k}$ and inside of $V_{k+1}$;
all possible  directed edges from $V_k$ to $V_{k+1}$ (but none from $V_{k+1}$ to $V_k$).
Moreover, we choose the classes so that 
$|V_i|=q+r_i$ for all $i \in [k-1]$, and  
$|V_i|=2q+r_i$ for all $i \in \{k,k+1\}$.

Let $v\in V_i$.  If $i\in [k-1]$, then $d_G(v)=2((k+2)q+r)-2r_i=2\ceil{(1-\frac{1}{k+3})n}-2r_i$.  If $i\in \{k,k+1\}$, then $d_G(v)=2((k+2)q+r)-r_j-2=2\ceil{(1-\frac{1}{k+3})n}-r_j-2$ where $j \in \{k,k+1\} \setminus \{i\}$. Therefore, $\delta (G)$ is as in the statement of the proposition.

Suppose for a contradiction that there is  a $k$th power of a Hamilton cycle~$C$ in~$G$. 
Notice that every transitive tournament on $k+1$ vertices in $G$ contains at least two vertices from~$V_k \cup V_{k+1}$.
By following the same ordering of vertices in~$V_k \cup V_{k+1}$ appearing in~$C$, we deduce that $C[V_k \cup V_{k+1}]$, and therefore $ G[V_k \cup V_{k+1}]$, contains a Hamilton cycle. 
However, $G[V_k \cup V_{k+1}]$ is not strongly connected, a contradiction. 
\end{proof}

\subsection{The extremal example for Theorem~\ref{thm:main}}\label{sec:ex}

Let $\vec{T}_k$ be the transitive tournament on $k$ vertices and let $\vec{r}(k)$ be the smallest $n\in \NN$ such that every $n$-vertex tournament  contains a copy of $\vec{T}_k$.
Let $\vec{tr}(k)$ be the smallest $n\in k\NN$ such that every $n$-vertex tournament has a $\vec{T}_k$-factor.

It is known that $\vec{r}(3)=4$, $\vec{r}(4)=8$, $\vec{r}(5)=14$, $\vec{r}(6)=28$, and $\lfloor \sqrt{2}^{k-1} \rfloor <\vec{r}(k) \le 2^{k-1}$.  Also $\vec{tr}(3)=6$, $\vec{tr}(4)=16$, and $ \vec{r}(k)\leq \vec{tr}(k)<4^k $ (see \cite[Section~5]{BS} for a comment about an improvement to this upper bound). 
We highlight that in all examples where $\vec{r}(k+1)$ is known, the lower bound example is a regular tournament. In particular, this is used in the proof of the following result.
\begin{proposition}\label{newprop}
Let $2 \leq k \le 5$. 
Given any $n \in \mathbb N$ divisible by $3\vec{r}(k+1)-1$, there exists an $n$-vertex oriented 
 graph $G_k$ with
\begin{align*}
\delta^0(G_k) \ge \left( 1-\frac{1}{3\vec{r}(k+1)-1} \right) \frac{n}2 - 2	
\end{align*}	
that does not contain the $k$th power of a Hamilton cycle.
\end{proposition}
\begin{proof}
Set $m:= \vec{r}(k+1)-1$ and
 let $n=(3m+2)t$ for some $t \in \mathbb N$. Consider the $n$-vertex oriented graph $G_k$ defined as follows (see also Figure~\ref{figureextremal} for~$k=2$).   
 The vertex set of $G_{k}$ consists of sets $V_1, V_2, V_3$ where either (i) $|V_1|= m t $ and $|V_2|=|V_3|=(m+1)t$ or (ii) $|V_1|= m t -1$, $|V_2|=(m+1)t+1$ and $|V_3|=(m+1)t$. We choose the sizes of the $V_i$ such that $|V_1|$ is not divisible by $k$. Add all edges directed from~$V_1$ to $V_{2}$, from $V_2$ to $V_{3}$, and from $V_3$ to $V_{1}$.
Both $G_k[V_2]$ and $G_k[V_3]$ are semi-regular tournaments.
Finally, let $G_k[V_1]$ be the $t$-blow-up of the regular tournament on $m$ vertices that contains no copy of $\vec{T}_{k+1}$ (where one vertex is deleted if $|V_1|=mt-1$).
It is easy to check that the desired minimum semi-degree condition holds.

Consider any  $k$th power of a  cycle $C$ in $G_k$. 
Notice that any copy of $\vec{T}_{k+1}$ in $G_k$ must contain vertices from at most two of the classes $V_1$, $V_2$, and $V_3$. This implies that every time $C$ enters $V_1$ (from $V_3$), it must traverse at least $k$ vertices before leaving $V_1$ (and entering $V_2$). In fact, since $G_k[V_1]$ does not contain a copy of $\vec{T}_{k+1}$, precisely $k$ vertices in $V_1$ are covered in each such step. Thus, the number of vertices in $V_1$ covered by $C$ is a multiple of $k$. Therefore, $C$ cannot contain all of $V_1$ since $|V_1|$ is not divisible by $k$.
\end{proof}
Note that Proposition~\ref{newprop} immediately implies Proposition~\ref{prop:example}.

\begin{figure}
\centering
\begin{tikzpicture}
\node (V1) at (0,0) [draw, circle, minimum size=3cm] {};
\node (V2) at (5,0) [draw, circle, minimum size=3.46cm] {};
\node (V3) at (2.5,-4.33) [draw, circle, minimum size=3.46cm] {};

\node (V1a) at (-0.75,0.6) [draw, circle,  minimum size=0.8cm] {};
\node (V2a) at (0.75,0.6) [draw, circle,  minimum size=0.8cm] {};
\node (V3a) at (0,-0.90) [draw, circle,  minimum size=0.8cm] {};

\draw[->, thick] (V1) -- (V2);
\draw[->, thick] (V2) -- (V3);
\draw[->, thick] (V3) -- (V1);

\draw[->, thick] (V1a) -- (V2a);
\draw[->, thick] (V2a) -- (V3a);
\draw[->, thick] (V3a) -- (V1a);

\draw[->, thick, decorate, decoration={ amplitude=.4mm, segment length=2mm}] (4.5,-0.5) arc[start angle=210, end angle=500, radius=0.75cm];
\draw[->, thick, decorate, decoration={ amplitude=.4mm, segment length=2mm}] (2,-4.83) arc[start angle=210, end angle=500, radius=0.75cm];

\node at (1.7,0.75) {$V_1$};
\node at (7,0.75) {$V_2$};
\node at (4.45,-3.58) {$V_3$};

\end{tikzpicture}
    \caption{The oriented graph $G_{2}$ does not contain a square of Hamilton cycle.}
    \label{figureextremal}
\end{figure}
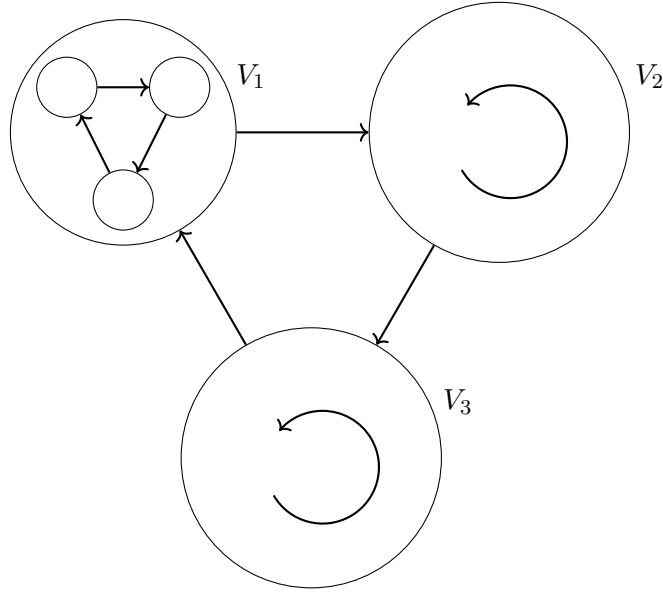

Next we prove the following general result which immediately yields the furthermore part of Theorem~\ref{thm:main}.  We will make use of the probabilistic construction of Erd\H{o}s and Moser \cite{EM} (which proves that $\vec{r}(k+1)> \floor{\sqrt{2}^k}$), combined with an additional calculation to show that a random tournament on $T$ vertices has minimum semi-degree very close to $T/2$.

\begin{proposition}\label{prop:growingk}
    For every~$k\geq 15$ and sufficiently large~$n\in \mathbb N$ there is an~$n$-vertex oriented graph~$R_k$ with
    $$\delta^0(R_k) > \bigg(\frac{1}{2}-\frac{4}{2^{k/5}}\bigg)n\,$$
    that does not contain a copy of~$\vec T_{k+1}$ and thus a $k$th power of a Hamilton cycle.
\end{proposition}

\begin{proof}
    Let $k\geq 15$ and set~$t:=\ceil{2^{(k-5)/2}}$.
    We first shall prove that there is  a (randomly generated) $t$-vertex tournament $T(k)$ with $\delta^0(T(k)) \geq  \big(\frac{1}{2}-\frac{3/2}{t^{2/5}}-\frac{1}{2t}\big)t$ that is $\vec{T}_{k+1}$-free. 
    Then we will get the desired oriented graph by taking a blow-up of such a tournament. 

    Consider a random tournament~$T$~on $t$ vertices; that is, the orientation of each edge is selected uniformly at random, independently of all other edges. By the union bound,
    \begin{align}\label{eq:noTk+1}
        \mathbb P(T\text{ contains a copy of }\vec{T}_{k+1})
        \leq 
        \frac{\binom{t}{k+1}(k+1)!}{2^{\binom{k+1}{2}}}
        <
        \frac{t^{k+1}}{2^{k(k+1)/2}}\,.
    \end{align}
    By Chernoff's bound,\footnote{If $X$ has a binomial distribution, then $\mathbb P(X\leq \mathbb{E}( X)-s)\leq e^{-s^2/(2\mathbb{E}(X))}$.}  for each $v \in V(T)$ we have $$\Prob\left(d^{\pm}_T(v)+\frac{1}{2}\leq \frac{t}{2}-(3/2)t^{3/5} = \left (\frac{1}{2} - \frac{3/2}{t^{2/5}}\right )t\right)\leq \exp\left (\frac{-(3/2)^2t^{6/5}}{t}\right ) = \exp\big(-(3/2)^2t^{1/5}\big).$$ 
    Then, by the union bound, we have  
    \begin{align}\label{eq:mindeg}
      \mathbb P
      \bigg(\delta^0(T)+\frac{1}{2} \leq \Big(\frac{1}{2} - \frac{3/2}{t^{2/5}}\Big)t\bigg)  
      \leq 
      2t\exp\big(-(3/2)^2t^{1/5}\big)\,. 
    \end{align}
    Thus, the probability that one of the events~\eqref{eq:noTk+1} or~\eqref{eq:mindeg} holds is at most
    $$\frac{t^{k+1}}{2^{k(k+1)/2}} 
    +
    2t\exp\big(-(3/2)^2t^{1/5}\big) 
     < 1\,,$$
    where the inequality holds for all $k\geq 15$ (in fact, one can see that the inequality holds for all $k\geq 2$, but we are assuming $k\geq 15$ in this context).
    Hence, there is a tournament~$T(k)$ on~$t$ vertices with~$\delta^0(T(k))> \big(\frac{1}{2} - \frac{3/2}{t^{2/5}}-\frac{1}{2t}\big)t$ not containing a copy of~$\vec T_{k+1}$. 

    Finally,~$R_k$ is obtained by blowing up each vertex of $T(k)$ into a set of size $\floor{\frac{n}{t}}$ or $\ceil{\frac{n}{t}}$, so that it contains~$n$ vertices in total.
    Thus,~$R_k$ does not contain a copy of $\vec{T}_{k+1}$ and 
    $$\delta^0(R_k)> \bigg(\frac{1}{2}-\frac{3/2}{t^{2/5}}-\frac{1}{2t}\bigg)t\cdot\left\lfloor\frac{n}{t}\right\rfloor
    \geq
    \bigg(\frac{1}{2}-\frac{3/2}{t^{2/5}}-\frac{1}{2t}\bigg)(n-(t-1))
    \geq \bigg(\frac{1}{2}-\frac{2}{t^{2/5}}\bigg)n
    \geq \bigg(\frac{1}{2}-\frac{4}{2^{k/5}}\bigg)n
    \,,$$
    as desired (where we used the fact that $n$ is sufficiently large in the second to last inequality and the fact that $t=\ceil{2^{(k-5)/2}}\geq 2^{(k-5)/2}$ in the last inequality).
\end{proof}

While Proposition \ref{prop:growingk} is sufficient for our proposes,
one can adapt the proof of Proposition~\ref{newprop} to obtain an oriented 
graph~$R_k'$ which does not contain the~$k$th power of a Hamilton cycle and has a higher minimum semi-degree. Indeed, one can adapt the construction in Proposition~\ref{newprop},  so that now  $V_1$ is spanned by the blow-up of a random tournament.
However, to obtain a better understanding of the minimum semi-degree threshold 
one would likely need a deeper understanding of the function~$\vec r(k)$.

\section{The regularity lemma and related results}\label{sec:pre}

The \emph{diregularity lemma} is a version of the regularity lemma for digraphs proved by Alon and Shapira \cite{AlonShapira04}. 
In this section we discuss the diregularity lemma and other related results that are needed for our proofs. 

We first require some notation.
Let $G$ be a digraph and $A,B\subseteq V(G)$ be disjoint. The \emph{density of~$(A,B)$} is defined by $d_G(A,B):=\tfrac{e_G(A,B)}{|A||B|}$. 
We will write $d(A,B)$ if this is unambiguous.
Note that~$d(A,B)$ is not necessarily equal to~$d(B,A)$.
Given~$\eps>0$  we say that~$(A,B)$ is~\emph{$\eps$-regular (in~$G$)} if for all subsets~$A'\subseteq A$ and~$B' \subseteq B$ with~$|A'|> \eps |A|$ and~$|B'|> \eps |B|$ we have $$|d_G(A,B)-d_G(A', B')|<\eps\,.$$
Finally, if $G=G[A, B]$ we write $G=(A,B)$.

We now state the  \emph{degree form} of the diregularity lemma. 

\begin{lemma}[Diregularity lemma~\cite{AlonShapira04}]\label{lem:reg}

Given any~$\eps\in (0,1)$ and~$t_0\in \mathbb N$, there exist $T=T(\eps,t_0) \in \mathbb N$ and $n_0=n_0(\eps, t_0)\in \mathbb N$ such that the following holds for all $n \geq n_0$. 
Let~$G$ be an $n$-vertex digraph   and let~$d\in[0,1]$. 
Then, there is a partition~$\{V_0, V_1,\dots, V_t\}$ of~$V(G)$ with~$t_0<t<T$ and a spanning subdigraph~$G'$ of $ G$ such that
\begin{enumerate}[label={\rm (\alph*)}]
    \item \label{it:trash}$|V_0|\leq \eps n$; 
    \item $|V_i|=|V_1|$ for every~$i\in [t]$; 
    \item for every $v \in V(G)$, $d^{+}_{G'}(v)>d^{+}_G(v)-(d+\eps)n$ and~$d^{-}_{G'}(v)>d^{-}_G(v)-(d+\eps)n$;
    \item $e(G'[V_i])=0$ for every~$i\in [t]$; 
    \item \label{it:reg} for every distinct~$i,j\in [t]$, the pair~$(V_i,V_j)$ is~$\eps$-regular in~$G'$ with density either~$0$ or at least~$d$. 
\end{enumerate}
\end{lemma}
We call $V_1,\dots, V_k$  \emph{clusters}, $V_0$ the \emph{exceptional set} and the vertices in $V_0$ \emph{exceptional
vertices}. We refer to $G_0$ as the \emph{pure digraph}. The last condition of Lemma~\ref{lem:reg} says
that all pairs of clusters are $\eps$-regular in both directions, but possibly with different
densities. The \emph{reduced digraph $R$ of $G$} with parameters $\eps$, $d$ and $t_0$ is the digraph
defined by
\begin{align*}
  V(R)&:=\{V_1,\dots, V_t\} \qquad \text{and} \qquad E(R):=\{V_iV_j \colon d_{G'}(V_i,V_j) \ge d   \}. 
\end{align*}

The following well-known result  states that the reduced digraph of $G$ essentially `inherits' any lower bound on the minimum total degree of $G$. 

\begin{prop}\label{prop:reduceddegree}
Let~$0<\eps\leq d/2$ and
    let $G$ be an $n$-vertex digraph such that $\delta (G)\geq \alpha n$ for some $\alpha >0$.
    Suppose we have applied Lemma~\ref{lem:reg} to $G$ to obtain the reduced digraph $R$ of $G$ with parameters $\eps$, $d$ and $t_0$. Then 
    $\delta(R)\geq (\alpha-4d)|R|\,.$\qed
\end{prop}

Note that when~$G$ is an oriented graph its reduced digraph~$R$ is not necessarily oriented (i.e., it may contain double edges). 
However, if for every double edge in $R$ we retain precisely one direction (with suitable probability), we obtain an oriented subgraph~$R_{o}$ of $R$ that, with positive probability, still inherits some of the properties of~$G$. 
This argument was formalised by Kelly, K\"uhn, and Osthus~\cite[Lemma 3.2]{KKO-diracoriented}, and in particular they proved the following result.

\begin{lemma}[{\cite[Lemma 3.2]{KKO-diracoriented}}]\label{prop:reducedOrientedDegree}
For every $\eps \in (0,1)$, if $t_0, n \in \mathbb N$ such that  $1/n \ll 1/t_0 \ll \eps$  then the following holds. Let $d, \alpha \in [0,1]$ and let $G$ be an $n$-vertex oriented graph  such that
$\delta^0(G) \geq \alpha n$.
Apply Lemma~\ref{lem:reg} to $G$ to obtain 
 the reduced digraph $R$ of $G$ with parameters $\eps$, $d$ and $t_0$. Then there is a spanning oriented subgraph~$R_o$ of $R$ such that
    $\delta^0(R_o)\geq (\alpha-(3\eps+d))|R_o|\,.$
\end{lemma}

Let $G$ be a graph and $A,B\subseteq V(G)$ be disjoint. We define the \emph{density} $d_G(A,B)$ analogously to before. 
As before, given~$\eps>0$  we say that~$(A,B)$ is~\emph{$\eps$-regular in~$G$} if for all subsets~$A'\subseteq A$ and~$B' \subseteq B$ with~$|A'|> \eps |A|$ and~$|B'|> \eps |B|$ we have $$|d_G(A,B)-d_G(A', B')|<\eps\,.$$
Given $\eps, d>0$, we say that $(A,B)$ is~\emph{$(\eps,d)$-superregular in~$G$} if 
$(A,B)$ is $\eps$-regular in~$G$
and, additionally,  $d_G(a) > d |B|$ for all $a \in A$ and  $d_G(b) > d |A|$ for all $b \in B$.

The next two propositions are well-known and easy to prove properties of regular pairs.
\begin{prop}\label{prop:regsubset}
Suppose that~$0<\eps < \xi \leq 1/2$. Let $G$ be a graph
and let $(A,B)$ be~$\eps$-regular in  $G$ with density~$d$. 
If~$A'\subseteq A$ and~$B'\subseteq B$ with~$|A'|\geq \xi |A|$ and $|B'|\geq \xi |B|$ then~$(A',B')$ is~$\eps/\xi$-regular in $G$ with density at least~$d-\eps$. \qed 
\end{prop}

\begin{prop}\label{prop:regdegree}
Given~$0<\eps<d\leq 1$, let $G$ be a graph and let~$(A,B)$ be $\eps$-regular  in  $G$ with density~$d$. 
There are at most~$\eps |A|$ vertices~$v\in A$ such that~$d_G(v, B)\leq (d-\eps)|B|$,
and at most $\eps |B|$ vertices~$w\in B$ such that~$d_G(w, A)\leq (d-\eps)|A|$.
\qed 
\end{prop}

We will use the following specific version of the blow-up lemma of Koml\'os, S\'ark\"ozy, and Szemer\'edi~\cite{blow}.

\begin{lemma}[Blow-up lemma~\cite{blow}] \label{lem:blowuplemma}
Let $1/m \ll \eps \ll 1/\ell,d,1/\Delta $.
Let $V_1, \dots, V_{\ell}$ be pairwise disjoint sets of vertices, each of size~$m$.
Let $R$ be a graph with $V(R) = \{V_1, \dots, V_\ell\}$.
Let $G$ be a graph on $V_1 \cup \dots \cup V_\ell$  such that $(V_i,V_j)$ is $(\eps, d)$-superregular in $G$ for each $V_iV_j \in E(R)$. 
Let $W_1,\dots, W_\ell$ denote the vertex classes of the 
$m$-blow-up $R(m)$  of~$R$ that correspond to $V_1,\dots, V_\ell$ respectively.
If $H$ is a subgraph of  $R(m)$  so that $\Delta(H) \le \Delta$,
then $G$ contains a copy of~$H$ such that, for each $i\in [\ell]$, the vertices in $V(H) \cap W_i$ are embedded into $V_i$ in $G$.
\end{lemma}
\begin{remark}\label{remark1}
 Note that, although Lemma~\ref{lem:blowuplemma} is stated for graphs, it is also applicable when $R$, $G$, and $H$ are {oriented} graphs such that all edges in $G[V_i\cup V_j]$ are oriented from $V_i$ to $V_j$ when $V_iV_j \in E(R)$. Indeed, in this case one can `ignore' the orientations of the edges and then apply  Lemma~\ref{lem:blowuplemma} to the underlying graphs of $R$, $G$, and $H$.   

 For example, suppose that $R$ is the $k$th power of the directed cycle $V_1\dots V_{\ell}V_1$ where $\ell \geq 2k+1$. 
 Then $R$ is an oriented graph. Thus, in the oriented graph $R(m)$ one can find the $k$th power of a Hamilton cycle by `winding around' the directed cycle $V_1\dots V_{\ell}V_1$. 
 Define $V_{\ell+1}:=V_1$.
Let $G$ be the oriented graph on $V_1 \cup \dots \cup V_\ell$ such that, for all $i \in [\ell]$, $G[V_i,V_{i+1}]$ induces an $(\eps, d)$-superregular pair 
$(V_i,V_{i+1})$ in the underlying graph of $G$. (To emphasize, importantly all edges between $V_i$ and $V_{i+1}$ in $G$ are oriented from $V_i$ to $V_{i+1}$.)
Then
 Lemma~\ref{lem:blowuplemma} tells us that $G$ contains the $k$th power of a Hamilton cycle.
Note that this argument relies on $R$ being an oriented graph (so does not work if $\ell \leq  2k$).
\end{remark}

We finish this section with the following embedding result that allows us to find long $k$-cycles (and therefore $k$-paths) in a digraph $G$ if its reduced digraph contains a $k$-cycle.

\begin{lemma}\label{lem:blowup}
Let $k,\ell , t_0 ,n \in \mathbb N$  and  $\eta, d, \eps >0$ be such that $k +1 \leq  \ell$ and $1/n \ll 1/t_0 \ll \eps \ll d, \eta, 1/\ell,1/k $.
Let $G$ be an $n$-vertex digraph and suppose $R$ is the reduced digraph of $G$ obtained by an application of Lemma~\ref{lem:reg}  with parameters~$\eps$, $d$ and~$t_0$. If 
$V_1 \dots V_{\ell}V_1$ is a copy of the $k$-cycle $C^k_{\ell}$ in~$R$, then there is a $k$-cycle in~$G[V_1 \cup \dots \cup V_{\ell}]$ covering all but at most $\eta \ell |V_{1}|$ vertices.
\end{lemma}

Recall that if  $ \ell \geq 2k+1$, then $C^k_{\ell}$ is an oriented graph. In this case of Lemma~\ref{lem:blowup}, one can apply Lemma~\ref{lem:blowuplemma} \'a la Remark~\ref{remark1} (with $C^k_{\ell}$ playing the role of $R$).
However, when $k+1 \leq \ell \leq 2k$, $C^k_{\ell}$ is not an oriented graph.
Thus, to apply Lemma~\ref{lem:blowuplemma} we first divide each cluster into two so that the corresponding reduced digraph contains a copy of~$C^k_{2\ell}$, which is an oriented graph since $2 \ell \geq 2k+1$.

\begin{proof}[Proof of Lemma~\ref{lem:blowup}]
Let $m$ be the largest integer such that $2 m \le |V_1|$; so $ 1 /m \ll \eps$. 
For each $i \in [\ell]$, let $W_i, W_{\ell+i}$ be disjoint subsets of~$V_i$ each of size~$m$.
Define the digraph $H_o :=\bigcup_{i \in [2\ell], \, j \in [k]}G[W_i,W_{i+j}]  $, where here the subindices are understood modulo~$2 \ell$. In fact, notice that crucially $H_o$ 
is an \emph{oriented} graph, and one can view $H_o$ as being obtained from a copy $W_1\dots W_{2\ell}W_1$ of~$C^k_{2\ell}$
by replacing each directed edge $W_iW_{i+j}$ in $C^k_{2\ell}$ with 
the oriented graph $G[W_i,W_{i+j}]$.

Let $H$ be the underlying graph of $H_o$.
Consider~$i \in [2\ell]$ and~$j \in [k]$.
By Proposition~\ref{prop:regsubset}, $(W_i,W_{i+j}) $ is $3\eps$-regular in~$H$ with density at least~$d-\eps$. 
By Proposition~\ref{prop:regdegree}, there are at most $3\eps m$ vertices~$v$ in each of $W_i$ and~$W_{i+j}$ such that $d_{H[W_i,W_{i+j}]} (v) \le (d -4 \eps)m $. 
Let $d'  := d - 6 ( k+1 ) \eps$ and $m' := (1 - 6k \eps)m  \ge m/2$.
There exists $W'_i \subseteq W_i$ of size~$m'$ for each $i \in [2\ell]$ such that, for each $i \in [2 \ell]$ and $j \in [k]$, $(W'_i,W'_{i+j}) $ is $( 6\eps, d')$-superregular in~$H$.
Indeed, this can be achieved by removing the $3 \eps m$ vertices in each of $W_i$ and~$W_{i+j}$ of the smallest degree in $H[W_i,W_{i+j}]$, for each $i \in [2\ell]$ and $j \in [k]$.
In particular, we have $\delta (H[W'_i,W'_{i+j}]) \ge (d -4 \eps)m - 2k \cdot 3 \eps m > d' m'$. 
Further, by Proposition~\ref{prop:regsubset},  $(W'_i,W'_{i+j}) $ is $6\eps$-regular in~$H$. 

Let $H' := H [W'_1 \cup \dots \cup W'_{2\ell}]$.
Let $R'$ be the graph on $\{W'_1, \dots,  W'_{2\ell}\}$ with $E(R') = \{ W'_iW'_{i+j} : i \in [2\ell] \text{ and } j \in [k] \}$; so $R'$ is an (undirected) copy of~$C^k_{2\ell}$.
Let $C^*$ be a copy of $C^k_{2\ell m' }$~in~$R'(m')$ obtained by `winding around'~$R'$.
We now apply Lemma~\ref{lem:blowuplemma} with $(6\eps, d', 2\ell,2k,H',R',C^*)$ playing the role of $ (\eps, d, \ell ,\Delta,G,R,H)$ to obtain a copy of  $C^k_{2\ell m' }$ in~$H$. 
This corresponds to an (oriented) copy of $C^k_{ 2\ell m' }$ in~$G[V_1 \cup \dots \cup V_{\ell}]$ covering all but at most $(1 + 12k \eps m )\ell  \le \eta \ell |V_1|$ vertices. 
\end{proof}

\section{Directed graphs: almost covering, absorbing, and connecting lemmas for Theorem~\ref{thm:square}}
\label{sec:p1}

As mentioned in Section~\ref{sec:overview}, the proof of Theorem~\ref{thm:square}
relies on three main auxiliary results: an  almost covering lemma, an absorbing lemma, and a connecting lemma. In this section, we prove these three results; see Lemmas~\ref{lem:almost}, \ref{lem:abs}, and~\ref{lem:manyCon}.

Recall that Theorem~\ref{thm:square} corresponds to the $k=2$ case of Conjecture~\ref{conjnew}.
Our almost covering lemma and  absorbing lemma actually hold for all~$k\geq 2$. 
In fact, the almost covering lemma requires a weaker minimum total degree condition than that in Conjecture~\ref{conjnew}. 
However, our connecting lemma only deals with the $k=2$ case. 
Consequently, 
the only ingredient missing for a full proof of the asymptotic version of Conjecture~\ref{conjnew} is a connecting lemma for~$k\geq 3$. We suspect though
that obtaining such a  connecting lemma will be rather challenging; we discuss this further in Section~\ref{sec:conc}.

\subsection{Almost Covering Lemma}

Given~$k\geq 3$, let~$\clique{k}$ denote the digraph obtained from the 
 complete digraph on~$k$ vertices by deleting a matching on~$\big\lfloor\tfrac{k}{2}\big\rfloor$ edges.  In particular, note that $C_{k+2}^{k}\subseteq \clique{k+2}$, but $C_{k+1}^{k}\not\subseteq \clique{k+1}$.\footnote{Note that $C_{k+1}^{k}$ is a complete digraph, whereas $C_{k+2}^{k}$ is obtained from a complete digraph by removing the edges of a Hamilton cycle.}
 \addtocounter{footnote}{-1}\addtocounter{Hfootnote}{-1}
 
The following result was proven in~\cite[Theorem 6.1]{cliquetilings}.

\begin{theorem}[Czygrinow, DeBiasio, Molla and Treglown \cite{cliquetilings}] 
\label{thm:cliquetiling}
    Given any~$k\geq 2$ and~$\eta>0$, there exists an~$n_0\in \mathbb N$ such that for every~$n\geq n_0$ the following holds. 
    If~$G$ is an~$n$-vertex digraph with
    $$\delta(G)\geq 2\left(1-\frac{1}{k} + \eta\right)n \,,$$
    then~$G$ contains a~$\clique{k}$-tiling covering all but at most $\eta n$ vertices.
\end{theorem}

We use Theorem~\ref{thm:cliquetiling} to prove the almost covering lemma.

\begin{lemma}[Almost covering lemma for total degree in digraphs]\label{lem:almost}
    Given any integer~$k\geq 2$ and~$\eta>0$, there exist $n_0,~T\in \mathbb N$ such that for any~$n\geq n_0$ the following holds. 
    If~$G$ is an~$n$-vertex digraph with
    $$\delta(G)\geq 2\left(1-\frac{1}{k+2}+\eta\right)n,$$
    then~$G$ contains a collection of at most $T$ vertex-disjoint $k$-paths that covers all but at most~$\eta n$ vertices. 
\end{lemma}

\begin{proof}
Define constants $\eps, \xi, d>0$ and $n_0,t_0,T \in \mathbb N$ such that 
 \begin{align}\label{eq:ctes-diralmost}
        \frac{1}{n_0} \ll \frac{1}{T} \leq \frac{1}{t_0} \ll \eps \ll \xi \ll d \ll \frac{1}{k}, \eta,
    \end{align}
and where $T$ is the output of Lemma~\ref{lem:reg} on input $\eps$ and $t_0$.

Let $G$ be a digraph on $n \geq n_0$ vertices as in the statement of the lemma.
Apply  Lemma~\ref{lem:reg} with parameters $\eps$, $d$ and $t_0$
to obtain the reduced digraph $R$ of $G$ with $t_0<|R|<T$.
Let $m$ denote the size of the clusters of $G$.
By Proposition~\ref{prop:reduceddegree} and (\ref{eq:ctes-diralmost}), 
    $$\delta(R)
    \geq
    2\left(1-\frac{1}{k+2}+\frac{\eta}{2}\right)|R|\,.$$
Thus, Theorem~\ref{thm:cliquetiling} yields a~$\clique{k+2}$-tiling $\mathcal T$ in~$R$ covering all but at most~$\eta |R| /2$  of the $V_i \in V(R)$.

 Recall that~$C_{k+2}^{k}\subseteq \clique{k+2}$.\footnotemark \
     Therefore, for each tile~$\clique{k+2}$ in $\mathcal T$ formed by clusters~$V_{i_1},\dots, V_{i_{k+2}}$,
    we may apply Lemma~\ref{lem:blowup} with~$\xi$ 
    playing the role of $\eta$, to obtain 
     a~$k$-path in~$G[\bigcup_{j\in [k+2]} V_{i_j}]$ covering all but at most~$\xi (k+2)m$ vertices from 
     $\bigcup_{j\in [k+2]} V_{i_j}$. 

     Together these $k$-paths  form a collection of size at most $|R|<T$. Moreover, all but at most 
     $$|V_0|+ \xi m |R|+\frac{\eta |R|}{2}m\leq (\eps +\xi +\eta/2)n \stackrel{(\ref{eq:ctes-diralmost})}{\leq} \eta n$$
vertices of $G$ are covered by these $k$-paths, as desired.
\end{proof}

\subsection{Absorbing Lemma}\label{sec:absdi}
In~\cite[Theorem 4.2]{cliquetilings},
the maximum number of edges in an~$n$-vertex digraph~$G$ without  a copy of~$\clique{k}$ was determined for all $k, n\in\mathbb{N}$. 
Using this result together with a standard supersaturation argument (e.g., via Lemma~\ref{lem:reg}), one can easily obtain the following theorem for the $t$-blow-up $\clique{k}(t)$.

\begin{theorem}\label{thm:directturan}
    Let $k\geq 2$,~$\eta >0$, and~$t\in\mathbb{N}$. 
    There exist~$n_0\in \mathbb N$ and $\xi>0$ such that every digraph~$G$ on~$n\geq n_0$ vertices with~$e(G)\geq \big(1-\frac{1}{k-1}+\eta \big)n^2$ contains at least~$\xi n^{kt}$ copies of~$\clique{k}(t)$.\qed
\end{theorem}

The following definition is a crucial notion needed for constructing an absorbing $k$-path.

\begin{definition}\label{defyabsy}
    For a digraph~$G$ and $v\in V(G)$, we say a~$k$-path on~$2(k+2)$ vertices~$P_v = v_1 \dots v_{2(k+2)}$ in~$G$ is an \emph{absorber for~$v$} if $v_1\dots v_{k+2}\, v\, v_{k+3} \dots v_{2(k+2)}$ is also a $k$-path  in $G$.
\end{definition}

We are now ready to state our absorbing lemma for digraphs. 

\begin{lemma}[Absorbing lemma for total degree in digraphs]\label{lem:abs}
    Given~$k\in \mathbb N$ and~$\eta>0$, there exist $n_0\in \mathbb N$ and~$\xi>0$ so that for any~$n\geq n_0$ the following holds. If~$G$ is an~$n$-vertex digraph with
    \begin{align*}
		\delta(G)\geq 2\Big(1 - \frac{1}{k+3} + \eta \Big) n \,,
		\end{align*}
    then, for every vertex~$v\in V(G)$, there are at least~$\xi n^{2(k+2)}$ absorbers~$P_v$ for $v$ in $G$.
\end{lemma}

\begin{proof}
Let~$\xi _1 >0$ be the output of  Theorem~\ref{thm:directturan} on input $k+2$, $\eta$ and $t=2$. Let $n_0 \in \mathbb N$ be sufficiently large and
 let~$G$ be a digraph on $n \geq n_0$ vertices as in the statement of the lemma.

Given any~$v\in V(G)$  note that
\begin{align*}
	|N^{+}_G(v)\cap N^-_G(v)| \geq 2 \left(1-\frac{1}{k+3}+\eta \right)n-n =  \left(\frac{k+1}{k+3}+2 \eta \right)n \,.
\end{align*} 
Let~$U$ be a subset of $N^{+}_G(v)\cap N^-_G(v)$ of size~$\tfrac{k+1}{k+3}n$.
Observe that
\begin{align*}
  \delta(G[U])  &\geq 2 |U| - 2 \left( \frac{1}{k+3}- \eta \right) n 
	= 2 \left( \frac{k}{k+3} + \eta \right)n  
	\geq    2\left(\frac{k}{k+1} + \eta \right)|U| \,,
\end{align*}
and in particular~$e(G[U])\geq \big(\tfrac{k}{k+1}+\eta)|U|^2$. 
Thus, Theorem~\ref{thm:directturan} yields at least~$\xi _1 |U|^{2(k+2)} \geq \xi _1 (n/2)^{2(k+2)} $ copies of~$\clique{k+2}(2)$ in~$G[U]$. Set $\xi:= \xi_1 /2^{2(k+2)}$.

As $C^{k}_{k+2} \subseteq \clique{k+2}$, it is easy to see that each such copy of $\clique{k+2}(2)$ in~$G[U]$ is spanned by a $k$-path. Moreover,  
as $U \subseteq N^{+}_G(v)\cap N^-_G(v)$, each of these
$k$-paths is an absorber for $v$, as desired.
\end{proof}

\subsection{Connecting Lemma} The full version of our connecting lemma (Lemma~\ref{lem:manyCon}) is stated at the end of this section. 
Before this, we introduce a few preliminary definitions and  results.

Given a digraph $G$, 
let $G^{\pm}$ be the graph with vertex set $V(G)$ such that $ab \in E(G^{\pm})$ if and only if $ab,ba \in E(G)$. 
For a vertex $v \in V(G)$ and $X \subseteq V(G)$, recall that $d_G(v,X) = d^+_G(v,X) + d^-_G(v,X)$. 
Given $U, W \subseteq V(G)$ and $m \in \NN$, define $N_{=m}(U, W) := \{ w \in W : d_G(w,U) =m \}$ and $d_{=m}(U,W) := |N_{=m}(U, W)|$.

The following result is the starting point for our connecting lemma.

\begin{lemma}\label{lem:partition}
  For \(1/n \ll \gamma < 1\) the following holds for every $n$-vertex digraph~$G$  with \(\delta(G) \ge (8/5 + \gamma)n\).
  If $\{A, B\}$ is a partition of $V(G)$ such that neither $A$ nor $B$ is an independent set, then there exist $a_1a_2 \in E(G[A])$,  $b_1b_2 \in E(G[B])$ and $xy \in E(G)$ such that $a_1a_2b_1b_2$ or $a_1a_2xyb_1b_2$ is a $2$-path.
\end{lemma}

\begin{proof}
 For a contradiction, assume that $G$ is an $n$-vertex digraph  and \( \{ A, B \}\) is a partition of \(V(G)\) that together form a counterexample to the lemma.
  By considering the digraph with all orientations reversed, we may assume without loss of generality that \(|A| \ge |B|\).
  Define 
  \begin{align*}
    E_A &:= \{ ab \in E(A, B) : \text{there exists \(a' \in A\) such that \(a'ab\) is a $2$-path} \},  \\
    E_B &:= \{ ab \in E(A, B) : \text{there exists \(b' \in B\) such that \(abb'\) is a $2$-path} \}.
  \end{align*}
	By our assumption, $E_A \cap E_B = \emptyset$.  
	Note that, for every $ab \in E(A, B)$, we have $d_{=4}(ab,V(G)) \ge 2\delta(G) - 3n \ge n/5 + 2\gamma n$ and so $ab \in E_A \cup E_B$. 
	Therefore, $\{E_A, E_B\}$ is a partition of $E(A, B)$.

  \begin{claim}\label{clm:EBnonempty}
    $E_B$ is non-empty.
  \end{claim}
  \begin{claimproof}
    Suppose not.
    Let $C$ be a maximal tournament in \(G[B]\).
    Since $B$ is not an independent set, we have $|C| \ge 2$.
    Note that $E_B = \emptyset$ implies that $d_G(a, C) \le |C| + 1$ for every $a \in A$.
    Furthermore, because $C$ is a maximal tournament in \(G[B]\), we have $d_G(b, C) \le 2|C| - 2$ for every $b \in B$.
    Recall that  $|A| \ge |B|$, so $|B| - 3n/5 < 0$.
		Hence, 
    \begin{align*}
      |C|\delta(G) 
      &\le \sum_{v \in C} d _G(v)
			\le |B|(2|C| - 2) + |A|(|C| + 1) 
      = |C|(n + |B|) - 2|B| + |A| \\
      &= 8n|C|/5 + |C|(|B| - 3n/5) - 2|B| + |A| \\
      &\le 8n|C|/5 + 2(|B| - 3n/5) - 2|B| + |A| < 8n|C|/5,
    \end{align*}
    a contradiction to the minimum total degree condition.
  \end{claimproof}

  By Claim~\ref{clm:EBnonempty} there exists $ab \in E_B$.  
	Set 
	\begin{align*}
	B_4 &:= N_{=4}(ab, B), & B_3 &:= N_{=3}(ab, B),  &A_3 &:= N_{=3}(ab, A), &c &:= |B_4| - n/5.
	\end{align*}
  Since $ab \in E_B$, we have $N_{=4}(ab, A) = \emptyset$ and  so 
  \[
    16n/5 + 2\gamma n \le d_G(a)+d_G(b) \le 2n + |B_3| + |A_3| + 2|B_4|.
  \]
  This implies that $c > 0 $ and $|B_3| + |A_3| \ge 4n/5 - 2c$.
  Since $|B_4| + |B_3| \le |B| \le n/2$ and $|B_4| = n/5 + c$, we have $|B_3| \le 3n/10 - c$ and $c \le 3n/10$, so $|A_3| \ge n/2 - c > 0$.
	Hence,
	\begin{align}
		|A_3 \cup B_4| \ge n/2 - c + n/5 + c = 7n/10. \label{eqn:A_3+B_4}
	\end{align}
	
  \begin{claim}\label{clm:A3B4empty}
    $E(A_3, B_4) = \emptyset$.
  \end{claim}
  \begin{claimproof}
    For a contradiction, assume that $xy \in E(A_3, B_4)$. 
    If $xa \in E(G)$, then $xayb$ is a $2$-path because $y \in N_{=4}(ab, B)$.
    If $xa \notin E(G)$, then the fact that $x \in N_{=3}(ab,A)$ implies that $ax,xb \in E(G)$, so $axby$ is a $2$-path. In both cases we obtain a contradiction to our initial assumption that $G$ is a counterexample to the lemma.
  \end{claimproof}

  Note that $\delta(G^{\pm}) \ge \delta(G) - n \ge 3n/5 + \gamma n$.
  Therefore, by \eqref{eqn:A_3+B_4} and Claim~\ref{clm:A3B4empty}, there exist $a_1a_2 \in E(G^{\pm}[A_3])$ and $b_1b_2 \in E(G^{\pm}[B_4])$.
  For $i \in \{7,8\}$ set $V_i := N_{=i}(\{a_1,a_2, b_1, b_2\}, V(G))$.
  Note that $32n/5 + 4 \gamma n \le 2|V_8| + |V_7| + 6n$, so 
	\begin{align*}
	2|V_8| + |V_7| \ge 2n/5 + 4 \gamma n.
	\end{align*}
  By Claim~\ref{clm:A3B4empty}, $V_7 \cup V_8$ is disjoint from $A_3 \cup B_4$, so 
  \[
    |V_7| + | V_8| \le n - |A_3 \cup B_4| \overset{\mathclap{\text{\eqref{eqn:A_3+B_4}}}}\le  3n/10.
  \]
  Therefore, 
  \[
    |V_8| = (2|V_8| + |V_7|) - (|V_7| + |V_8|) \ge n/10 + 4 \gamma n.
  \]
  Let $x \in V_8$. 
	Since we have assumed that $G$ is a counterexample to the lemma, in $G$,  $x$ has no neighbors in $V_8$ and,  in $G^{\pm}$, $x$ has no neighbors in $V_7$. Thus,
  \[
    \delta(G) \le d_G(x) \le 2n - 2|V_8| - |V_7| < 8n/5, 
  \]
  a contradiction.
\end{proof}

\begin{definition}
In a digraph $G$, 
  a $2$-walk of length $\ell$  is a sequence $v_1 \dotsc v_\ell$ of vertices from $V(G)$ such that, for every $i \in [\ell - 1]$ we have $v_iv_{i+1} \in E(G)$ and, moreover, for every
  $j \in [\ell - 2]$ we have $ v_jv_{j+2} \in E(G)$.
  Note that the vertices $v_1, \dotsc, v_\ell \in V(G)$ are not necessarily distinct.
\end{definition}
The notion of a $2$-walk is used in the proof of the following result.

\begin{lemma}\label{dtc}
Let $0 < 1/n \ll \gamma < 1$.
  If $G$ is an $n$-vertex digraph  with $\delta(G)\geq (8/5+\gamma)n$, then for every 
  pair of disjoint edges $ab, yz \in E(G)$, there exists a $2$-path $x_1x_2 \dots x_{\ell-1}x_{\ell}$ of length $\ell \leq 20$ where $x_1=a$, $x_2=b$, $x_{\ell-1}=y$, and $x_\ell =z$.
\end{lemma}

\begin{proof}
  Let $\sigma > 0$ be such that {$1/n\ll \sigma\ll\gamma$}.
  Take $G$, $ab$, and $yz$  as in the statement of the lemma. 
  For $\ell \ge 3$, let 
  \[
    \fromab{V}_\ell := \{v \in V(G) : \text{$\exists$ at least $(\sigma n)^{\ell - 3}$  $2$-walks of length $\ell$ that start with $ab$ and end with $v$}\}
  \]
  and 
  \[
    \toyz{V}_\ell := \{v \in V(G) : \text{$\exists$ at least $(\sigma n)^{\ell - 3}$  $2$-walks of length $\ell$ that start with $v$ and end with $yz$}\}.
  \]
  \begin{claim}\label{clm:n5reachable}
    For every $\ell \ge 3$, we have $|\fromab{V}_\ell|, |\toyz{V}_\ell| \ge n/5$.
  \end{claim}
  \begin{claimproof}
    We will only show that $|\fromab{V}_\ell| \ge n/5$ since the proof that $|\toyz{V}_\ell| \ge n/5$ follows analogously.

		Let $\Phi_{2}: = \{ab\}$.
    For $t \ge 3$, let $\Phi_{t}$ be the collection of $2$-walks of length $t$ in $G$  that start with the edge $ab$. 
    Since $\delta^+(G) \ge (3n/5 + \gamma) n$, for all $t \geq 3$ we have that 
		\begin{align}
		|\Phi_{t}| \ge (1/5 + 2\gamma )n |\Phi_{t - 1}| \ge (1/5 + 2\gamma )^{t-2}n^{t-2}. \label{eqn:Phi_t}
		\end{align}
		For a contradiction, suppose that $|\fromab{V}_\ell| < n/5$.
		Hence the number of $2$-walks in $\Phi_{\ell}$ that end with a vertex in~$\fromab{V}_\ell$ is less than $|\Phi_{\ell-1}|n/5$.
		So the number of $2$-walks in $\Phi_\ell$ that end at a vertex not in~$\fromab{V}_\ell$ is at least
		\begin{align*}
		|\Phi_{\ell}| - |\Phi_{\ell-1}|n/5 &
		\overset{\mathclap{\text{\eqref{eqn:Phi_t}}}}{\ge} 2\gamma n |\Phi_{\ell - 1}|
		\overset{\mathclap{\text{\eqref{eqn:Phi_t}}}}{\ge} 2\gamma n (1/5 + 2\gamma )^{\ell-3}n^{\ell-3}
		\ge \sigma^{\ell-3} n^{\ell-2}.
		\end{align*}
		By an averaging argument, there exists a vertex $v \in V(G) \setminus \fromab{V}_\ell$ that is  the end vertex of at least 
		$\sigma^{\ell-3} n^{\ell-3}$ $2$-walks in~$\Phi_{\ell}$, a contradiction. 
  \end{claimproof}

  Let $\{V^+, V^-\}$ be the partition of $V(G)$ such that for every $v \in V^-$ we have $d^-_G(v) \ge d^+_G(v)$ and for every $v \in V^+$ we have $d^+_G(v) > d^-_G(v)$.
  \begin{claim}\label{clm:V-reachable}
    For $\ell \ge 5$, we have $V^- \subseteq \fromab{V}_\ell$ and $V^+ \subseteq \toyz{V}_\ell$.
  \end{claim}
  \begin{claimproof}
    We will only show that $V^- \subseteq \fromab{V}_\ell$ as the proof that $V^+ \subseteq \toyz{V}_\ell$ is analogous.
    Let $v \in V^-$. By Claim~\ref{clm:n5reachable} and the fact that $v \in V^-$, we have 
    \[
      d^-_G(v, \fromab{V}_{\ell - 2}) \ge \delta(G)/2 + n/5 - n \ge \gamma n/2.
    \]
    Therefore, there are at least $(\sigma n)^{\ell - 5}\times \gamma n/2 \geq (\sigma n)^{\ell - 4}$  $2$-walks $abv_3 \dotsc v_{\ell-3} v_{\ell-2}$ in $G$ with $v_{\ell-2} \in N^-_G(v)$. 
    Furthermore, for every such $2$-walk we have
    \[
      |N^+_G(v_{\ell-3}) \cap N^+_G(v_{\ell-2}) \cap N^-_G(v)| \ge 2(\delta(G) - n) + \delta(G)/2 - 2n \ge 2\gamma n,
    \]
    and for every $v_{\ell - 1} \in N^+_G(v_{\ell-3}) \cap N^+_G(v_{\ell-2}) \cap N^-_G(v)$, 
    the sequence
    \[
      abv_3 \dotsc v_{\ell-3} v_{\ell-2} v_{\ell-1} v
    \] 
    is a $2$-walk in $G$. Thus, there are  at least $(\sigma n)^{\ell - 4} \times 2 \gamma n \geq (\sigma n)^{\ell - 3}$  $2$-walks in $G$ that start with $ab$ and end at $v$.
    Hence, $v \in \fromab{V}_\ell$ and so $V^- \subseteq \fromab{V}_\ell$, as desired.
  \end{claimproof}

  For $\ell \ge 4$, define
  \[
    \fromab{E}_\ell := \{e \in E(G) : \text{$\exists$ at least $(\sigma n)^{\ell - 4}$ $2$-walks of length $\ell$ that start with $ab$ and end with $e$}\}.
  \]
  and 
  \[
    \toyz{E}_\ell := \{e \in E(G) : \text{$\exists$ at least $(\sigma n)^{\ell - 4}$  $2$-walks of length $\ell$ that start with $e$ and end with $yz$}\}.
  \]

	Let $X^- \subseteq V^-$ be such that $|X^-| = \min \{\sigma n, |V^-|\}$ and $d^+_G(x) \ge d^+_G(y)$ for every $x \in X^-$ and $y \in V^- \setminus X^-$.
	Similarly, let $X^+ \subseteq V^+$ be such that $|X^+| = \min \{\sigma n, |V^+|\}$ and $d^-_G(x) \ge d^- _G(y)$ for every $x \in X^+$ and $y \in V^+ \setminus X^+$.
	Set $Y^- := V^- \setminus X^-$ and $Y^+ := V^+ \setminus X^+$.

  \begin{claim}\label{clm:E-reachable}
    For every $\ell \ge 9$, 
    $E(G[Y^-]) \subseteq \fromab{E}_\ell$ and  $E(G[Y^+]) \subseteq \toyz{E}_\ell$.
  \end{claim}
  \begin{claimproof}
    We will only show that $E(G[Y^-]) \subseteq \fromab{E}_\ell$ since the proof that $E(G[Y^+]) \subseteq \toyz{E}_\ell$ is analogous.
    We may assume that $E(G[Y^-]) \neq \emptyset$ and so 
    $|X^-| = \sigma n$.
    Let $cd \in E(G[Y^-])$.
    Pick $v_{\ell-4} \in X^-$;
    by Claim~\ref{clm:V-reachable}, $v_{\ell-4} \in \fromab{V}_{\ell-4}$.
    Therefore, there are at least $(\sigma n)^{\ell-7}$  $2$-walks $W$ of length $\ell-4$ in $G$ that start with $ab$ and end with~$v_{\ell-4}$.
		Let $v_{\ell-5}$ be the penultimate vertex on any such $2$-walk $W$.
    There are more than $2(\delta(G) - n) + \delta(G)/2 - 2n \ge 2\gamma n$ vertices in $N^+_G(v_{\ell-5}) \cap N^+_G(v_{\ell-4}) \cap N^-_G(c)$.
    Let $v_{\ell-3}$ be such a vertex.
		
   There are at least $\gamma n$ vertices~$v_{\ell-2}$ in $N^+_G(v_{\ell-4}) \cap N^+_G(v_{\ell-3}) \cap N^-_G(c) \cap N^-_G(d)$. 
    To see this, recall that $v_{\ell - 4} \in X^-$ and $c,d \in Y^- \subseteq V^-$, so $d^+_G(v_{\ell-4}) \ge d^+_G(c)$ and $d^-_G(d) \ge \delta(G)/2$.
    Therefore,
    \begin{align*}
		& | N^+_G(v_{\ell-4}) \cap N^+_G(v_{\ell-3}) \cap N^-_G(c) \cap N^-_G(d)| 
		\ge d^+_G(v_{\ell-4}) + d^+_G(v_{\ell-3}) + d^-_G(c) + d^-_G(d) - 3n \\
    & \ge  d^+_G(v_{\ell-4}) + (\delta(G) - n) + (\delta(G) - d^+_G(c)) + \delta(G)/2 -3n \ge 5\delta(G)/2 - 4n \ge   \gamma n.
    \end{align*}
	In summary, there are  $\sigma n$ choices for $v_{\ell-4}$; at least  $(\sigma n)^{\ell-7}$  choices for the $2$-walk $W$; at least $2 \gamma n$ choices for $v_{\ell-3}$; at least $\gamma n $ choices for $v_{\ell-2}$.
 By the choice of $v_{\ell-3}$ and $v_{\ell-2}$, adding $v_{\ell-3}v_{\ell-2}cd$ to the end of $W$ yields a $2$-walk of length $\ell$ in $G$ that starts with $ab$ and ends with $cd$. In total, this process gives rise to at least
$$\sigma n \times (\sigma n)^{\ell-7} \times 2 \gamma n \times \gamma n \geq (\sigma n)^{\ell - 4} $$
 such $2$-walks. Thus, $cd \in \fromab{E}_\ell$ and so
$E(G[Y^-]) \subseteq \fromab{E}_\ell$, as desired.
  \end{claimproof}

 Let $Y^-_1:=Y^- \setminus \{a, b, y, z\}$ and
 $Y^+_1:=Y^+ \setminus \{a, b, y, z\}$.
  Suppose that $Y^-_1$ is an independent (or empty) set in $G$.
	Note that $|Y^-_1 \cup Y^+_1| \ge (1- 2 \sigma)n-4$ and $\delta^0(G) \ge  \delta (G) - n \ge (3/5+\gamma) n $.
	Thus, there exist $v_1,v_2,v_3,v_4 \in Y^-_1 \cup Y^+_1$ such that $abv_1v_2v_3v_4$ is a $2$-path in $G$.
	Since $Y^-_1$ is an independent (or empty) set, there exists some $i \in [3]$ such that $v_i,v_{i+1} \in Y^+_1$.
	By Claim~\ref{clm:E-reachable}, $v_i v_{i+1} \in E(G[Y^+]) \subseteq \toyz{E}_{12-i}$.
	Therefore, there are at least $(\sigma n)^{8-i}$  $2$-walks of length $13$ in $G$ that start with $abv_1 \dots v_iv_{i+1}$ and end with $yz$.
	By a simple counting argument, one such $2$-walk is in fact a $2$-path. Thus, the conclusion of the lemma holds in this case.

 An analogous argument holds in the case when 
   $Y^+_1$ is an independent (or empty) set in $G$.
Thus, we may assume that neither $Y^-_1$ nor $Y^+_1$ is an independent or empty set. 

 Let $G' := G \setminus ( X^- \cup  X^+ \cup \{a, b, y, z\})=G[Y^-_1 \cup Y^+ _1]$, so $\delta(G') \ge \delta(G) - 2 \sigma n - 4 \ge (8/5+\gamma/2) n $.
  Then by applying Lemma~\ref{lem:partition}  to~$G'$ and the partition $\{Y^-_1, Y^+_1\}$ of $V(G')$, there exists a $2$-path on at most $6$ vertices in $G'$ that {starts} with an edge in $E(G[Y^-_1])$ and ends with an edge in $E(G[Y^+_1])$.
  By Claim~\ref{clm:E-reachable}, 
  $E(G[Y^-_1]) \subseteq \fromab{E}_9$ 
  and 
$E(G[Y^+_1]) \subseteq \toyz{E}_9$.
Combining these facts, and using an averaging argument, one obtains a 
 $2$-path on at most $20$ vertices in $G$ that starts with~$ab$ and ends with~$yz$, as desired.
\end{proof}

From Lemma~\ref{dtc} we can easily deduce the following slight strengthening.

\begin{lemma}[Connecting lemma for total degree in digraphs]\label{lem:manyCon}
Let $0 < 1/n \ll  \gamma < 1$.
  If $G$ is an $n$-vertex digraph  with $\delta(G)\geq (8/5+\gamma)n$, then for every 
  pair of disjoint edges $ab, yz \in E(G)$
  and every set~$U\subseteq V(G)\setminus\{a,b,y,z\}$ of size at most $\gamma n/2$,
  there exists a $2$-path $x_1x_2 \dots x_{\ell-1}x_{\ell}$ of length $\ell \leq 20$
  in $G \setminus U$ where $x_1=a$, $x_2=b$, $x_{\ell-1}=y$, and $x_\ell =z$.
\end{lemma}

\begin{proof}
    As $\delta(G \setminus U)\geq (8/5 + \gamma/2)n$, this follows immediately by applying  Lemma~\ref{dtc} to $G\setminus U$. 
\end{proof}

\section{Oriented graphs:  absorbing, connecting, and almost covering lemmas for Theorem~\ref{thm:square}}\label{sec:p2}

The goal of this section is to present the proofs of the  absorbing, connecting, and almost covering lemmas for Theorem~\ref{thm:main};  see Lemmas~\ref{lem:abs-oriented},~\ref{lem:manyCon-oriented}, and~\ref{lem:tiling}.

\subsection{Auxiliary lemmas for connecting and absorbing}
\label{subsec:Lemmas}

In contrast to Lemma~\ref{lem:manyCon}, our connecting lemma for oriented graphs (Lemma~\ref{lem:manyCon-oriented}) does not {imply} the existence of a $k$-path between every pair of~$k$-tuples~$T_1$ and~$T_2$. 
It will only imply the existence of such a~$k$-path if~$T_1$ has many common out-neighbors and~$T_2$ has many common in-neighbors. 
To make this precise, we introduce the following definition. 

\begin{definition}\label{inoutgood}
    Let $G$ be an $n$-vertex oriented graph and $0 <\delta <1 $. 
    We say that a $k$-vertex tournament~$T$ in $G$ is \emph{$\delta$-out-good} if $|\bigcap_{x\in V(T)}N^+_G(x)|\geq \frac{(2k-1)\delta-k+1}{k2^{2k-1}}n$.  
    Likewise,~$T$ is \emph{$\delta$-in-good} if $|\bigcap_{x\in V(T)}N^-_G(x)|\geq \frac{(2k-1)\delta-k+1}{k2^{2k-1}}n$.
    If a~$k$-path~$P$ starts with a $\delta$-in-good tournament on $k$ vertices and ends in a $\delta$-out-good tournament on $k$ vertices, we say that~$P$ is a \emph{$\delta$-good~$k$-path}.
	
\end{definition}
Note that  Definition~\ref{inoutgood} is only useful when $\delta $ is reasonably large, since for  $0<\delta \leq (k-1)/(2k-1)$, the common out- and in-neighborhood conditions here do not even guarantee a single common out- or in-neighbor.

The next lemma states that every  large enough tournament in a dense oriented graph contains a $\delta$-in/out-good subtournament on $k$ vertices. 

\begin{lemma}
\label{lem:goodTk}
Let $k\geq 2$ and let $G$ be a sufficiently large $n$-vertex oriented graph  with $\delta^0(G) = \delta n > \frac{(k-1)n}{2k-1}$. 
Let $T$ be a $(2k-1)$-vertex tournament in~$G$.
Then $T$ contains a $\delta$-out-good tournament and a $\delta$-in-good tournament, each on $k$ vertices. 
\end{lemma}

\begin{proof}
Let  $X:=V(T)$, $Y:=\{v\in V(G): d^-_G(v,X)\leq k-1\}$, and $Z:=\{v\in V(G): d^-_G(v,X)\geq k\}$.
We have 
\begin{align*}
	(\delta n-2k+1) (2k-1)& = (\delta n-|X|) |X|
	\leq e_G (X, V(G)\setminus X) \\
	& \leq (k-1)|Y|+(2k-1)|Z| = (k-1)n+k|Z|,
\end{align*}
so
\begin{align*}
 |Z| & \geq \frac{( (2k-1)\delta -k+1 )n-(2k-1)^2}{k}.
\end{align*} 
Thus, there exists a $k$-set $X'\subseteq X$ such that 
\begin{align*}
	\left|\bigcap_{x\in X'}N^+_G(x) \right| & 
	\geq \frac{|Z|}{\binom{2k-1}{k}}
	\geq \frac{ ( (2k-1)\delta -k+1)n-(2k-1)^2}{\binom{2k-1}{k}k}
	\geq \frac{(2k-1)\delta -k+1}{k2^{2k-1}}n\,,
\end{align*}
where the last inequality follows as $n$ is sufficiently large. 
Therefore,~$X'$ induces a $\delta$-out-good tournament on $k$ vertices.  
The proof for $\delta$-in-good tournaments  is analogous.
\end{proof}

The following lemma states that given an oriented graph $G$ with large minimum semi-degree, for every two large disjoint sets of vertices~$A$ and $B$, there are many edges from~$A$ to $B$.

\begin{lemma}[Crossing edges between large sets]
\label{lem:crossbound}
Let $G$ be an $n$-vertex oriented graph with $\delta^0(G)\geq \delta n \geq  \frac{2n}{5}$.  
Then, for any disjoint sets $A,B\subseteq V(G)$, we have 
\begin{align*}
	e_G(A,B)\geq |A|\left(\frac{|A|}{2}+|B|-(1-\delta)n\right).
\end{align*}
In particular, if~$|A|\geq |B| \geq \delta n$, then
\begin{align*}
e_G(A,B)\geq \left( \frac{5}{2} - \frac{1}{\delta} \right)|A||B|.
\end{align*}
\end{lemma}

\begin{proof}
We have 
\begin{align*}
	e_G(A,B)&\geq \sum_{v\in A} \left( d^+_G(v) - d^+_G(v, A) - (n-|A|-|B|) \right) 
\geq |A|\delta n-e(G[A])-|A|(n-|A|-|B|)\\
&\geq |A|\delta n-\frac{|A|^2}{2}-|A|(n-|A|-|B|)
= |A|\left(\frac{|A|}{2}+|B|-(1-\delta)n\right).
\end{align*}
When~$|A|\geq |B| \geq \delta n$, the desired inequality follows easily since
\begin{align*}
 |A|\left(\frac{|A|}{2}+|B|-(1-\delta)n\right) 
	= |A||B| \left(\frac{|A|}{2|B|}+1-\frac{(1-\delta)n}{|B|}\right) 
	\geq |A||B| \left(\frac{5}{2} - \frac{1}{\delta}\right).
\end{align*}
\end{proof}

Recall that $\vec{r}(k)$ and $ \vec{tr}(k)$ are defined at the beginning of Section~\ref{sec:ex}.
The following result will allow us to find large transitive tournaments in oriented graphs with large minimum semi-degree.

\begin{lemma}\label{lem:turHS}
Let $k \geq 2$ and let $G$ be an $n$-vertex oriented graph with $\delta^0(G)=\delta n$.  
\begin{enumerate}[label={\rm(\roman*)}]
\item If $X\subseteq V(G)$ with $|X|>(\vec{r}(k)-1)(1-2\delta)n$, then $G[X]$ contains a copy of $\vec{T}_k$.
\item If $X\subseteq V(G)$ with $|X|\geq \vec{tr}(k)(1-2\delta)n$ and $|X|$ divisible by $k$, then $G[X]$ has a $\vec{T}_k$-factor.
\item If $X\subseteq V(G)$ with $|X|\geq 3^k(1-2\delta)n$, then $G[X]$ contains at least $\left(\frac{|X|}{3^{k}}\right)^k$ copies of $\vec{T}_k$.
\end{enumerate}
\end{lemma}

\begin{proof}
Note  $|X|>(\vec{r}(k)-1)(1-2\delta)n$ implies that $\delta(G[X])\geq |X|-(1-2\delta)n>\frac{\vec{r}(k)-2}{\vec{r}(k)-1}|X|$. Thus, by Tur\'an's theorem there exists a tournament $T$ on $\vec{r}(k)$ vertices in $G[X]$; by definition of $\vec{r}(k)$, $T$ contains a copy of $\vec{T}_k$. 

Similarly, $|X|\geq \vec{tr}(k)(1-2\delta)n$ implies that $\delta(G[X])\geq |X|-(1-2\delta)n\geq \frac{\vec{tr}(k)-1}{\vec{tr}(k)}|X|$. By~\cite[Proposition~9]{treg} (itself a simple corollary of the Hajnal--Szemer\'edi theorem), $G[X]$ contains a $\vec{T}_k$-factor. 

For (iii), note  $|X|\geq 3^k(1-2\delta)n$ implies that $\delta(G[X])\geq |X|-(1-2\delta)n\geq (1-\frac{1}{3^k})|X|$.  The idea is to greedily construct a $\vec{T}_k$ in $G[X]$ by choosing an arbitrary vertex and noting that at least half of its incident edges have the same direction.  Then look inside that neighborhood, choose an arbitrary vertex and repeat. 
More precisely, for all $0\leq i\leq k-2$ we have $\frac{\frac{|X|}{3^i}-\frac{|X|}{3^k}}{2}\geq \frac{|X|}{3^{i+1}}$, and thus this greedy process produces at least $\prod_{i=0}^{k-1}\frac{|X|}{3^i}=\frac{|X|^k}{3^{k(k-1)/2}}$ copies of $\vec{T}_k$ in $G[X]$; however, up to $2^k$ 
 different options for this greedy process give rise to the same copy of $\vec{T}_k$. Therefore, 
we obtain  at least $\frac{1}{2^k\cdot 3^{k(k-1)/2}}|X|^k\geq \frac{1}{3^{k^2}}|X|^k$ copies of $\vec{T}_k$ in $G[X]$. 
\end{proof}

The following lemma is an amalgamation of \cite[Lemma 2.9]{DMS} and \cite[Lemma 6.3]{FS2}.  Since the statement is slightly different, for completeness we rewrite their proof tailored to our statement.

\begin{lemma}[Dependent random choice variant]
\label{lem:DRC}
Let $k \in \mathbb N$, $0<d\leq 1$, and define $c := d^k/\sqrt[k]{2}$.
Let $G=(A,B)$ be a bipartite graph with $e(G)\geq d |A||B|$. For all
$0<\ep<1$, there exists $U\subseteq A$ with $|U| \geq c |A|$ such that all but at most $\left(\ep |U|\right)^k$ of the  $k$-tuples in $U$ have at least~$\ep c |B|$ common neighbors in $B$.
\end{lemma}

\begin{proof}
Let $S$ be a subset of $k$ random vertices, chosen uniformly from $B$ with repetition. Let $U$ denote the set of common neighbors of $S$ in $A$. Note that linearity of expectation and Jensen's inequality imply   
$\mathbb{E}[|U|] = \sum_{v\in A} \left(\frac{d_G(v)}{|B|}\right)^k\geq d^k |A| = 2^{1/k} c |A|.$

Let $Y$ denote the number of $k$-tuples in $U$ with fewer than $m:= \ep c |B|$ common neighbors in $B$. Note that, by linearity of expectation,
\[
\mathbb{E}[Y] < |A|^k \left(\frac{m}{|B|} \right)^k =\left(\ep c |A|\right)^k.
\]
By the previous two inequalities and another application of Jensen's inequality, we have
\[
\EE[Y]+ \left(\ep c |A| \right)^k
< 2 \left(\ep c|A| \right)^k =  \left(\ep \left(2^{1/k} c|A|\right)\right)^k \le \left(\EE[\ep |U|]\right)^k\leq \EE\left[\left(\ep |U|\right)^k\right],
\]
which, by linearity of expectation, implies
$$\EE\left[\left(\ep |U|\right)^k -Y- \left( \ep c |A|\right)^k\right]\geq 0.$$
Thus, there is a choice of $S$ for which $Y \leq \left(\ep |U|\right)^k$ and for which $ (\ep |U|)^k \ge \left(\ep c |A| \right)^k$, and so $|U|\geq c |A|$. 
\end{proof}

\subsection{The absorbing lemma}
  
In this section we will use a very slightly different  absorber compared to that used in Section~\ref{sec:absdi}. 
\begin{definition}\label{absdefy}
  Given a digraph $G$
and~$x\in V(G)$, we say that a~$k$-path $u_1\dots u_kv_1\dots v_k$ on~$2k$ vertices in $G$ is a \textit{$k$-absorber for~$x$} if~$u_1\dots u_kxv_1\dots v_k$ is a~$k$-path in $G$ as well.
We say that a $k$-path  $u_1\dots u_{2k}v_1\dots v_{2k}$ on~$4k$ vertices in $G$ is a \textit{stretched $k$-absorber for~$x$} if~$u_1\dots u_{2k}xv_1\dots v_{2k}$ is a~$k$-path in $G$ as well.
\end{definition}
Note that we will only use the notion of a stretched $k$-absorber in the proof of Theorem~\ref{thm:main} in Section~\ref{sec:71}.

\begin{lemma}\label{lem:abs-oriented}
Let $0< 1/n \ll \xi \ll 1/k \le 1/2$. 
If $G$ is an $n$-vertex oriented graph  with $\delta^0(G) =\delta n\ge \left(\frac{1}{2}-\frac{1}{4\cdot 3^{3k+2}}\right)n$,
then every $x\in V(G)$ has at least $\xi n^{2k}$ $k$-absorbers in $G$.
\end{lemma}

\begin{proof}
Given any $x \in V(G)$,
let $A\subseteq N^{-}_G(x)$ and~$B\subseteq N^+_G(x)$ be both of size~$\delta n$. 
By Lemma \ref{lem:crossbound}, we have 
\begin{align*}
	e_G(A,B)\geq \Big(\frac{5}{2} - \frac{1}{\delta}\Big)|A||B| \geq \frac{1}{3}|A||B|\,.
\end{align*} 
Let~$U\subseteq A$ be the set obtained by applying
 Lemma~\ref{lem:DRC} to the underlying graph of $G[A,B]$ and  with parameters
\begin{align*}
	d:=1/3 \text{ and }\eps :=1/3^{k+1}.
\end{align*}
In particular, it holds that
\begin{align*}
    |U|\geq \frac{d^k}{2}|A| =\frac{\delta n}{2\cdot 3^{k}} \geq \frac{ n}{2\cdot 3^{k+1}} \ge 3^k(1-2\delta) n \, . 
\end{align*}
Thus, Lemma~\ref{lem:turHS}(iii) yields at least~$(|U|/3^k)^k$ copies of $\vec{T}_k$ in $G[U]$. 

Due to our application of Lemma~\ref{lem:DRC}, all but at most $(|U|/3^{k+1})^{k}$ of the~$k$-tuples in~$U$ have at least~$\eps d^k|B|/2 =|B|/(2\cdot 3^{2k+1})$ common out-neighbors in~$B$. 
Hence, there are at least~$(3^{-k^2}- 3^{-k^2-k})|U|^k> (3^{-k^2-1})|U|^k$ copies $T_1$ of $\vec{T}_k$ in $G[U]$ such that the vertices in $T_1$ have at least~$|B|/(2\cdot 3^{2k+1})$ common out-neighbors in~$B$. 

Fix one such tournament~$T_1\subseteq G[U]$ and observe that
\begin{align*}
	\left| \bigcap_{v \in V(T_1)} N^{+}_G(v, B) \right| \geq \frac{|B|}{2\cdot 3^{2k+1}} \geq \frac{n}{2\cdot 3^{2k+2}} \geq 3^k(1-2\delta) n \,.
\end{align*}
An application of Lemma~\ref{lem:turHS}(iii) yields at least~$(|\bigcap_{v \in V(T_1)} N^{+}_G(v,B)|/3^k)^k\geq n^k /(3^{3k(k+1)} )$ 
copies~$T_2$ of $\vec{T}_k$ contained in~$\bigcap_{v \in V(T_1)} N^{+}_G(v,B)$. 
Fix one such tournament~$T_2$. As $V(T_1)\subseteq U \subseteq A \subseteq  N^- _G(x)$ and
$V(T_2)\subseteq B \subseteq N^+ _G(x)$,
 we can let~$T_1=:u_1\dots u_k$ and~$T_2=:v_1\dots v_k$ so that  $u_1\dots u_kv_1\dots v_k$ is a $k$-absorber for $x$.

 Note that there are more that $(3^{-k^2-1})|U|^k\geq (3^{-k^2-1})\frac{d^{k^2}}{2^k}|A|^k\geq \frac{n^k}{3^{3k^2+1}}$ choices for $T_1$ and, given a fixed choice of $T_1$, at least $n^k /(3^{3k(k+1)} )$ choices for $T_2$.
 Thus, in total we obtain at least
 $$\frac{n^k}{3^{3k^2+1}} \times \frac{n^k }{3^{3k(k+1)}  }\geq \xi n^{2k}$$ $k$-absorbers for $x$, as desired.
\end{proof}


\subsection{The connecting lemma}

The following  lemma allows us to connect a 
$\delta$-out-good copy of $\vec{T}_{k}$ to a  
$\delta$-in-good copy of $\vec{T}_{k}$ by a short $k$-path that avoids any small set of vertices.

\begin{lemma}[Connecting an out-good $\vec{T}_{k}$ to an in-good $\vec{T}_{k}$]
\label{lem:manyCon-oriented}
Let $0<1/n \ll \zeta \ll 1/k \leq 1/2$. 
Let $G$ be an $n$-vertex oriented graph  with $\delta^0(G) \ge \left(\frac{1}{2}-\frac{1}{ 3^{18k}} \right)n$.
Set $\delta :=\left(\frac{1}{2}-\frac{1}{2\cdot 3^{17k}} \right)$. 
Given any pair of vertex-disjoint $\delta$-out-good $T^+$ and $\delta$-in-good $T^-$ copies of $\vec{T}_{k}$ in $G$, and any 
 set of vertices~$U\subseteq V(G)\setminus (V(T^{+}) \cup V(T^-))$ of size at most~$\zeta n$, there exists $\ell \in \{3,4\}$ such that there is a $k$-path from $T^+$ to $T^-$ on $k(\ell+2)$ vertices in $G\setminus U$. 
\end{lemma}

\begin{proof}
Set~$\alpha:=1/27$ and 
$\tau  := {3^{-6(k+1)}}$.  
Note that 
\begin{align}\label{eq:connect:ctes}
\tau \leq \frac{1}{45}
\text{, }\quad
16 \cdot \frac{\tau}{\alpha^{2k}} \leq \frac{19}{45}
\text{, }\quad \text{and}\quad
3^k(1-2\delta) = \frac{1}{3^{16k}} \leq \frac{\alpha^{2k}\tau}{\sqrt[k]{4}} \leq  \frac{\alpha^{2k}}{16 \cdot 8^k} \,. 
\end{align}
Take~$G$,~$T^{+}$,~$T^{-}$, and~$U$ to be as in the statement of the lemma. As $\zeta \ll 1/k$ we have $\delta^0(G\setminus U)\geq \delta n$.

Set~$G':=G\setminus U$ and let
\begin{align*}
	A_1 &:= \left (\bigcap_{v \in V(T^+)} N^+_{G'}(v) \right ) \setminus  V(T^-), &
	B_1 &:= \left (\bigcap_{v \in V(T^-)} N^-_{G'}(v) \right ) \setminus V(T^{+}) , 
\end{align*}
	$$
 A_2 := \{ v \in V({G'})\setminus  (V(T^{+}) \cup V(T^-)) : d_{G'} ^-(v, A_{1}) \ge \alpha |A_{1}| \} ,
 $$ and
 $$
	B_2 := \{ v \in V({G'})  \setminus (V(T^{+}) \cup V(T^-)): d_{G'} ^+(v, B_{1}) \ge \alpha |B_{1}| \} .
$$
So $A_1, A_2, B_1$ and $B_2$ are all disjoint from $T^+$ and $T^-$.
As $T^+$ is $\delta$-out-good in $G$, we have that
\begin{align*}
    |A_1| & \geq 
    \frac{ (2k-1)\delta-k+1}{k2^{2k-1}}n -|U|-k
    \geq \frac{\frac{1}{2} - \frac{2k-1}{2\cdot 3^{17k}}}{k2^{2k-1}} \cdot n -2\zeta n
    \ge 
    \frac{n}{4k \cdot 2^{2k-1}}
    \ge
    \frac{n}{8^k}.
\end{align*}
Furthermore,
\begin{align*}
 \delta n |A_{1}| \leq \sum_{v\in A_{1}} d^+_{G'}(v)  \leq |A_{1}|(|A_2|+2k)+ \alpha |A_{1}| (n-|A_2|).
\end{align*} 
Therefore, we have~$|A_2|\geq\frac{\delta-\alpha}{1-\alpha}n-\frac{2k}{1-\alpha} \geq \tfrac{4}{9}n$.
Analogous calculations for $|B_1|$ and $|B_2|$ imply that
\begin{align}\label{eq:sizeX2Y2}
    |A_1|,|B_1|  \ge  \frac{n}{8^k}
		 \text{ and }
		|A_2|,|B_2| \geq \frac{4}{9}n.
\end{align}

We will split into two cases depending on the size of $A_2 \cap B_2$.
First we prove the following claim. 

\begin{claim}\label{clm:goodtour}
Let~$Y,Z \subseteq V({G'})$  be such that~$|Y|\geq \tau n$ and $d^+_{G'}(y,Z) \ge \alpha |Z|$ for all $y \in Y$.
Then there is a copy $T$ of $\vec{T}_{k}$ in~${G'}[Y]$ such that 
\begin{align*}
		\left|\bigcap_{v \in V(T) } N^+_{G'}(v,   Z) \right| \geq 	\frac{\alpha^{2k}}{16}|Z|\,
	. 
\end{align*}
Moreover, if there exists $X \subseteq V({G'}) $ such that  $e_{G'}(X,Y) \geq \alpha |X||Y|$, then we may further assume that 
\begin{align*}
 \left|\bigcap_{v \in V(T) } N^-_{G'}(v,  X) \right| \geq 	\frac{\alpha^{2k}}{16}|X|\, .
\end{align*}
\end{claim}
Note that in this claim, $X$, $Y$ and $Z$ are not necessarily disjoint.

\begin{claimproof}
We shall apply Lemma~\ref{lem:DRC} with~$\alpha$ playing the role of~$d$ and~$\eps :=\alpha^k/(4\sqrt[k]{4})$, and so we set~$c:=\alpha^k/\sqrt[k]{2}$.
If we are not in the moreover case of the claim, then we set $W:= Y$.
In the moreover case, we  apply Lemma~\ref{lem:DRC} to the bipartite graph $H$ which has 
vertex classes $X$ and $Y$ and an edge between $x \in X$ and $y \in Y$ precisely if the directed edge $xy$ is present in $G'$.
Thus, we  obtain a set~$W\subseteq Y$ 
 such that $|W| \ge c |Y| $ and in $G'$ all but at most~$(\eps |W|)^k$ $k$-tuples in~$W$ have at least~$\eps c |X|$ common in-neighbors in~$X$.

Next
we define $H'$ to be the bipartite graph 
 which has 
vertex classes $W$ and $Z$ and an edge between $w \in W$ and $z \in Z$ precisely if the directed edge $wz$ is present in $G'$.
Since $W \subseteq Y$, we have $e(H') \ge \alpha |W||Z|$. 
So applying Lemma~\ref{lem:DRC}  to $H'$ we obtain a subset~$\widetilde W\subseteq W$ where
\begin{align} \label{eqn:goodtour_U}
	| \widetilde{W} | \ge c|W| \ge c^2 |Y| \geq \alpha^{2k} \tau n / \sqrt[k]{4}  \, 
    \overset{\mathclap{\text{\eqref{eq:connect:ctes}}}}{\ge} \,
	3^k(1-2 \delta)n,
\end{align}
and in $G'$ all but at most~$(\eps |\widetilde W|)^k \leq (\eps |W|)^k$ $k$-tuples in~$\widetilde W$ have at least~$\eps c |Z|$ common out-neighbors in~$Z$.

Lemma~\ref{lem:turHS} and~\eqref{eqn:goodtour_U} imply that there are at least~$(|\widetilde W|/3)^k 
\geq c^k |W|^k/3^k 
> 2(\eps |W|)^k$ copies $T$ of $\vec{T}_{k}$ in~$G'[\widetilde W]$.
Recalling that we took~$\eps =\alpha^k/(4\sqrt[k]{4})$ and $c=\alpha^k/\sqrt[k]{2}$, we obtain one such $T$ with
\begin{align*}
	\left|\bigcap_{v \in V(T) } N^+_{G'}(v,  Z) \right| \geq 	\eps c |Z| \geq \frac{\alpha^{2k}}{16}|Z| 
\end{align*} 
and, in the moreover case 
\begin{align*}
 \left|\bigcap_{v \in V(T) } N^-_{G'}(v ,  X )\right| \geq \frac{\alpha^{2k}}{16}|X| ,
\end{align*} 
as required. 
\end{claimproof}

We first assume that $|A_2 \cap B_2| \le \tau n $.
Let $A_2' := A_2 \setminus B_2$ and $B_2' := B_2 \setminus A_2$; so 
\begin{align}\label{eq:sizeAB'}
|A_2'|, |B_2'| \stackrel{\eqref{eq:sizeX2Y2}}{\ge} \frac{4n}9 - \tau n \stackrel{\eqref{eq:connect:ctes}}{\geq} \frac{19n}{45} \stackrel{\eqref{eq:connect:ctes}}{\geq} 16 \frac{\tau n}{\alpha^{2k}}.
\end{align}
By Lemma~\ref{lem:crossbound} (with $19/45$ playing the role of $\delta$), we have 
\begin{align*}
      e_{G'}(A_2',B_2') \ge \left( \frac52 - \frac{45}{19} \right)|A_2'||B_2'| \ge \alpha |A_2'||B_2'|.
\end{align*}
By Claim~\ref{clm:goodtour} with $A_2',B_2',B_1$ playing roles of $X,Y,Z$, we obtain a copy~$T_{B_2}$ of $\vec{T}_{k}$ in~$G'[B'_2] \subseteq G'[B_2]$ 
such that 
\begin{align*}
		\left|\bigcap_{v \in V(T_{B_2}) } N^+_{G'}(v,  B_1 )\right| & \geq 	\frac{ \alpha^{2k}}{16}|B_1| 
		 \stackrel{\eqref{eq:sizeX2Y2}}{\ge}
   \frac{ \alpha^{2k}}{16 \cdot 8^k}n
   \stackrel{\eqref{eq:connect:ctes}}{\ge}
			3^k(1-2 \delta)n  \, 	\text{ and }\\
 \left|\bigcap_{v \in V(T_{B_2}) } N^-_{G'}(v,  A_2' )\right| &\geq 	\frac{\alpha^{2k}}{16}|A_2'| \stackrel{\eqref{eq:sizeAB'}}{\ge} \tau n\, .
\end{align*}
By applying Claim~\ref{clm:goodtour} with $\bigcap_{v \in V(T_{B_2}) } N^-_{G'}(v, A_2' )$ and $ A_1$ playing the roles of $Y$ and $ Z$ (and with the orientations of each edge reversed), we obtain 
a copy~$T_{A_2}$ of $\vec{T}_{k}$ 
 in $G'[\bigcap_{v \in V(T_{B_2}) } N^-_{G'}(v, A'_2)]$ such that 
\begin{align*}
 \left|\bigcap_{v \in V(T_{A_2}) } N^-_{G'}(v,  A_1) \right| \geq 	\frac{ \alpha^{2k}}{16}|A_1|  \stackrel{\eqref{eq:sizeX2Y2},\eqref{eq:connect:ctes}}{\ge}
			3^k(1-2 \delta)n\, .
\end{align*}
Since $A'_2 \cap B'_2 =\emptyset$, we have that $T_{A_2}$ and $T_{B_2}$ are disjoint.

By Lemma~\ref{lem:turHS}(iii), there is a copy~$T_{B_1}$ of $\vec{T}_{k}$ in $G'[\bigcap_{v \in V(T_{B_2}) } N^+_{G'}(v,  B_1) ]$ that is disjoint from  both $T_{A_2}$ and $T_{B_2}$.
Similarly, Lemma~\ref{lem:turHS}(iii) implies that there is a copy~$T_{A_1}$ of
$\vec{T}_{k}$ 
in~$G'[\bigcap_{v \in V(T_{A_2}) } N^-_{G'}(v,  A_1) ]$
that is disjoint from $T_{A_2}$, $T_{B_1}$, and $T_{B_2}$.
Note that the concatenation of $T^+, T_{A_1}, T_{A_2}, T_{B_2}, T_{B_1} ,T^-$ yields a $k$-path from~$T^+$ to~$T^-$
in $G'$ on $6 k$ vertices (with some additional edges), as desired.

We now outline the proof of the case when $|A_2 \cap B_2| \ge \tau n $.
By Claim~\ref{clm:goodtour} with $A_1,A_2 \cap B_2,B_1$ playing roles of $X,Y,Z$, we obtain a copy $T_{A_2 \cap B_2}$ of $\vec{T}_{k}$  in $G'[A_2 \cap B_2]$. In fact, we can argue similarly to the previous case to find a copy 
$T_{A_1}$ of $\vec{T}_{k}$  in $G'[A_1]$ and a copy 
$T_{B_1}$ of $\vec{T}_{k}$  in $G'[B_1]$, so that the  
concatenation of  $T^+, T_{A_1}, T_{A_2\cap B_2}, T_{B_1} ,T^-$ yields a $k$-path from~$T^+$ to~$T^-$
in $G'$ on $5 k$ vertices, as desired.
\end{proof}


 \subsection{The almost covering lemma}

Our goal in this subsection is to prove the following result. 

\begin{lemma}\label{lem:tiling}
Let $k\geq 2$ and let $Q\geq 10^{1000k}$. 
There exists $n_0\in \mathbb N$ such that if $G$ is an oriented graph on $n\geq n_0$ vertices with $\delta^0(G)\geq (\frac{1}{2}-\frac{1}{(20Q)^3})n$, then $V(G)$ can be partitioned into sets $S$ of size $Q$ or $Q+1$,
where for each such $S$ we have that $G[S]$ is a tournament, $\delta ^0(G[S]) \geq 2|S|/5$ and $G[S]$
 contains the $k$th power of a Hamilton cycle. 
In particular, $G$ can be partitioned into at most $n/Q$ vertex-disjoint $k$-cycles. 
\end{lemma}

Bollob\'as and H\"aggkvist~\cite{BH} proved that for all $\ep>0$ and $k\in \mathbb N$, there exists $n_0\in \mathbb N$ such that for every tournament $T$ on $n\geq n_0$ vertices, if $\delta^0(T)\geq \frac{n}{4}+\ep n$, then $C^k_n\subseteq T$.  Dragani\'c, Munh\'a Correia, and Sudakov~\cite[Theorem 1.5]{DMS} gave a refinement of this result,  which in particular gives better quantitative bounds.  
We state their result in a less general form which is more convenient for our purposes.

\begin{theorem}\label{thm:BoHa}
Let $k\geq 2$ and let $n\geq 10^{1000k}$.  If $T$ is an $n$-vertex tournament  with $\delta^0(T)\geq 2n/5$, then $T$ contains the $k$th power of a Hamilton cycle.
\end{theorem}

Therefore, to prove Lemma~\ref{lem:tiling}, we will first partition~$G$ into vertex-disjoint tournaments each of size $Q$ or $Q+1$ and with minimum semi-degree at least $2(Q+1)/5$.
Then the lemma follows by applying Theorem~\ref{thm:BoHa} to each tournament. 
We need the following lemma on martingales.  

\begin{lemma}[{\cite[Lemma 2.2]{ABHKP}}]
  \label{lem:coupling} Let $\Omega$ be a finite probability space and
  let $\F_0, \dots, \F_n$ be partitions of~$\Omega$,
  with $\F_{i-1}$ refined by $\F_i$ for each $i \in [n]$.
  For each $i\in[n]$, let $Y_i$ be a Bernoulli random
  variable on $\Omega$ that is constant on each part of~$\F_i$.  Furthermore, let~$p_i$ be a
  real-valued random variable on~$\Omega$ 
  which is constant on each part of $\F_{i-1}$.
  Let~$x$ and~$\delta$ be  real numbers with $\delta\in(0,3/2)$, and
  let $X:=Y_1+\dots+Y_n$.  If $\sum_{i=1}^n p_i\ge x$ holds almost
  surely and $\EE [Y_i\mid\F_{i-1}]\ge p_i$ holds almost surely for
  all $i\in[n]$, then
  $\PP\big(X<(1-\delta)x \big)<e^{-\delta^2 x/3}\,.$
\end{lemma}

\begin{proof}[Proof of Lemma~\ref{lem:tiling}]
Given such a~$Q$, let $\gamma := (20Q)^{-3}$ and let $n_0\in \mathbb N$ be sufficiently large. 
Let $G$ be  an oriented graph as in the statement of the lemma.
We randomly partition $V(G)$ into $t: = \floor{n/Q}$ sets $S_1, S_2, \dots, S_t$ each of size~$Q$ and $n-tQ\leq Q-1$ other vertices.
Then each $S_i$ can be viewed as a uniformly random set of~$Q$ vertices from~$G$.
We claim that with positive probability, all three of the following properties hold:
\begin{enumerate}[label={\rm (\alph*)}]
	\item \label{itm:tiling1}
	at most $4 \gamma Q^2 t $ of the $G[S_i]$ are not tournaments;
	\item \label{itm:tiling2}
	at most $ 4 e^{-Q/20^4} t$ of the $G[S_i]$ have minimum semi-degree below  $2(Q+1)/5$;
	\item \label{itm:tiling3}
	for every $v\in V(G)$, there are at least $t/200$ indices $i \in [t/100]$ such that $d^+_G(v, S_i),  d^-_G(v, S_i)\ge 2(Q+1)/5$.
\end{enumerate}
We now show that this implies the lemma. 
Let $S_1, \dots, S_t$ be such that~\ref{itm:tiling1}--\ref{itm:tiling3} holds. 
Let $I$ be the set of $i \in [t]$ such that $G[S_i]$ is not a tournament or $\delta^0(G[S_i]) < 2(Q+1)/5$; so 
\begin{align*}
	|I| \le 4 \gamma Q^2 t + 4 e^{-Q/20^4} t  \le t/(500Q).
\end{align*}
Let $W  := V(G) \setminus \bigcup_{i \in [t] \setminus I} S_i$, so 
\begin{align*}
	|W| \le Q |I| + Q-1  < t/400.
\end{align*}
Recall that every vertex has at most $2\gamma n$ non-neighbors in~$G$.
Together with~\ref{itm:tiling3}, for each $w \in W$, the number of $i \in [t/100] \setminus I$
such that $S_i \subseteq N^+_G(w)\cup N^-_G(w)$ and  $d^+_G(w, S_i), d^-_G(w, S_i)\ge 2(Q+1)/5$ is at least
\begin{align*}
	t/200 - |I| - 2\gamma n \ge t/400 > |W|.
\end{align*}
Therefore, for each $w \in W$, we can greedily assign it a unique $i\in[t/100] \setminus I$ such that $G[S_i\cup \{w\}]$ is a tournament of order $Q+1$ with minimum semi-degree at least $2(Q+1)/5$. 
The lemma holds by applying Theorem~\ref{thm:BoHa} to each tournament.

\smallskip

To complete the proof, we will show that each of \ref{itm:tiling1}--\ref{itm:tiling3} fails with probability at most~$1/4$.

Firstly, note that $G$ has at most $(2\gamma n-1)n/2\leq 2 \gamma \binom{n}{2}$ non-edges.
For $i \in [t]$, let $X_i$ be the number of non-edges in $G[S_i]$.
Then we have 
\begin{align*}
    \EE X_i\le 2\gamma \binom Q2\le \gamma Q^2.
\end{align*}
For $i \in [t]$, let $E_i$ be the event that $G[S_i]$ is not a tournament, namely, $G[S_i]$ has some non-edges; let $X$ be the number of $i \in [t]$ such that $G[S_i]$ is not a tournament.
Therefore, we obtain 
\begin{align*}
    \EE X = 
    \sum_{i\in [t]} 1 \cdot \PP (E_i) 
    = \sum_{i\in [t]} \PP (X_i \ge 1) 
    \le \sum_{i\in [t]}  \EE (X_i) \le \gamma Q^2 t\,.
\end{align*}
By Markov's inequality, we obtain that
\begin{align*}
\PP ( \text{\ref{itm:tiling1} fails} ) = 
\PP (X > 4 \gamma Q^2 t) \le 1/4.
\end{align*}

Secondly, consider $i \in [t]$ and $v \in S_i$. 
Then
by Chernoff's bound for the hypergeometric distribution, we obtain that
\[
\PP (d^+_G(v, S_i)\le 2(Q+1)/5)  , \, \PP (d^-_G(v, S_i)\le 2(Q+1)/5) \le e^{-Q/20^3}.
\]
So, by the union bound, for every $i \in [t]$ we have
\[
\PP\left(\delta^0(G[S_i]) < \frac{2(Q+1)}{5}\right) \le 2Qe^{-Q/20^3} \le 
e^{-Q/20^4}.
\]
Let $Y$ be the number of $i\in [t]$ such that $\delta^0(G[S_i]) < 2(Q+1)/5$; so $\EE Y\le t e^{-Q/20^4}$.
By Markov's inequality, we obtain that
\begin{align*}
\PP ( \text{\ref{itm:tiling2} fails} ) = 
\PP (Y > 4te^{-Q/20^4}) \le 1/4.
\end{align*}

Finally, fix $v\in V(G)$ and let $t_0 := t/100$.
Consider a process of picking a sequence of vertex-disjoint sets $S_1, S_2, \dots, S_{t_0}$ of size $Q$, one by one, each time uniformly at random from the remaining vertices.
Note that this is equivalent to considering the first $t_0$ members of our partition $S_1, \dots, S_t$.
Now condition on any outcome of $S_1, \dots, S_{i-1}$.
Let $S:=\bigcup_{j\in [i-1]} S_j $ and thus $|S|= (i-1)Q< n/100$.
Then we have $d^+_G(v, V(G)\setminus S) , \,  d^-_G(v, V(G)\setminus S) \ge (0.49-\gamma)n$.
Therefore, by Chernoff's bound for the hypergeometric distribution, we obtain that
\begin{align*}
\PP (d^+_G(v, S_i)\le 2(Q+1)/5 \mid S_1,  \dots, S_{i-1}), \, \PP (d^-_G(v, S_i)\le 2(Q+1)/5 \mid S_1, \dots, S_{i-1})  \le e^{-Q/20^3} \,. 
\end{align*}
For $i\in [t_0]$, let $E_{v, i}$ be the event that $d^+_G(v, S_i), \, d^-_G(v, S_i)\ge 2(Q+1)/5$ holds.
Then we get 
\begin{align*}
\EE (\mathbf{1}_{E_{v, i}}\mid S_1, \dots, S_{i-1})= \PP(E_{v, i}\mid S_1, \dots, S_{i-1}) \ge 1 - 2e^{-Q/20^3} \ge 2/3.
\end{align*}
Let $Z_v:=\sum_{i\in [t_0]} \mathbf{1}_{E_{v, i}}$. 
So by Lemma~\ref{lem:coupling} with $(n,\mathcal{F}_i,Y_i,p_i,x,\delta,X) = (t_0, S_i, \mathbf{1}_{E_{v, i}},2/3,2t_0/3,1/4 , Z_v)$, we obtain that
\[
\PP (Z_v < t_0/2) \le  e^{-t_0/72}.
\]
By the union bound, 
\begin{align*}
\PP ( \text{\ref{itm:tiling3} fails} ) \le \sum_{v \in V(G) }\PP (Z_v < t_0/2) \le n e^{-t_0/72} \le 1/4,
\end{align*}
where the last inequality follows as $t_0 = \lfloor n/Q \rfloor /100$ and $n$ is sufficiently large. This completes the proof of the lemma. 
\end{proof}

\section{Proofs of Theorems~\ref{thm:square} and~\ref{thm:main}}\label{sec:6}
In this section we combine our auxiliary lemmas to prove both 
 Theorem~\ref{thm:square} and  Theorem~\ref{thm:main}.
 In the next subsection we prove Theorem~\ref{thm:main}.
The proof of Theorem~\ref{thm:square} is quite standard, and follows the same structure as Theorem~\ref{thm:main};  so we do not provide all of the calculations in the proof in Section~\ref{sec:sketch}.

\subsection{Proof of Theorem~\ref{thm:main}}\label{sec:71}
The second part of Theorem~\ref{thm:main} follows from Proposition~\ref{prop:growingk}, so we just need to prove the first part. Let {$c  :=  10^{6000}$} and $k \ge 2$. 
Define constants $\eps, \eta, d, \zeta, \xi>0$ and $n_0, T \in \mathbb N$  so that
\begin{align*}
0< 1/n_0 \ll 1/T\ll \eps \ll\eta \ll d \ll \zeta,\xi \ll 1/c,1/k.
\end{align*}
Given any $n \geq n_0$, let~$G$ be an $n$-vertex oriented graph with
\begin{align*}
	\delta^0(G) \geq \left (\frac{1}{2}-\frac{1}{c^k} \right )n.
\end{align*}
Let  $\delta : =\frac{1}{2}-\frac{2}{c^k}$.
Throughout the proof, we simply write good, in-good and out-good to mean $\delta$-good, $\delta$-in-good and $\delta$-out-good (with respect to $G$), respectively. 

\smallskip
\textbf{Constructing the absorbing path.}
We first find a good~$k$-path~$P_A$ whose first and last~$k$ vertices are denoted by~$A_1$ and~$A_2$ respectively, and such that
\begin{enumerate}[label={\rm (\roman*)}]
    \item $|V(P_A)|\leq \eta^{k} n$ and\label{it:abs-size}
    \item for every set~$L\subseteq V(G)\setminus V(P_A)$ of size at most~$\eta^{2k}n$, $G[L \cup V(P_A)]$ contains a spanning~$k$-path from~$A_1$ to~$A_2$. \label{it:abs-global}
\end{enumerate}

To do this, for each~$x\in V(G)$ define
$$\mathscr A_x := \{P \colon P \text{ is a $(3k-1)$-absorber for~$x$}\}\,,$$
and set~$\mathscr A:=\bigcup_{x\in V(G)}\mathscr A_x$.\footnote{Recall that a $(3k-1)$-absorber consists of $2(3k-1)$ vertices; see Definition~\ref{absdefy}.}
By Lemma~\ref{lem:abs-oriented} (with $3k-1$ playing the role of $k$), we have~$|\mathscr A_x|\geq \xi n^{2(3k-1)}$.
Let~$\mathcal C\subseteq \mathscr A$ be a random collection of $(3k-1)$-absorbers in which we independently include each element of $\mathscr A$ with probability~$p:=4\eta^{4k/3}n/|\mathscr A|$.
Since~$\mathbb E(|\mathcal C|) = p|\mathscr A| = 4\eta^{4k/3}n$, Markov's inequality yields that
\begin{align}\label{eq:sizeofAbs}
    \mathbb P(|\mathcal C|\geq 8\eta^{4k/3}n) \leq 1/2\,.
\end{align}
Moreover, for every~$x\in V(G)$ we have
\begin{align*}
\mathbb E(|\mathcal C\cap \mathscr A_x|) = p|\mathscr A_x| \geq 4\eta^{4k/3}n \cdot \xi n^{2(3k-1)}/|\mathscr A| \geq 4\eta^{5k/3}n.
\end{align*}
Thus, Chernoff's bound together with the union bound yields that
\begin{align}\label{eq:manyabs}
    \mathbb P\big(\exists x\in V(G) \colon |\mathcal C\cap \mathscr A_x| \leq 2\eta^{5k/ 3}n\big) 
    \leq 
    2n \exp(-\eta^{5k/3}n/3) < 1/4 \, .
\end{align}
Finally, let~$X$ be the random variable that counts the number of pairs of~$(3k-1)$-paths~$P,P'\in\mathcal C$ sharing at least one vertex.
Since the number of pairs of~$(3k-1)$-paths on~$2(3k-1)$ vertices that intersect in at least one vertex is at most~$(6k-2)^2n^{12k-5}$, we have
$$\mathbb E(X) 
\leq 
p^2 4(3k-1)^2 n^{12k-5}
=
\frac{16 (3k-1)^2\eta^{8k/3}n^{12k-3}}{|\mathscr A|^2}
\leq 
\frac{16 (3k-1)^2\eta^{8k/3}}{\xi ^2}\cdot n 
\leq \eta^{2k}n \,,$$ 
where for the second inequality we used~$|\mathscr A|\geq \xi n^{6k-2}$.
Again by Markov's inequality we obtain
\begin{align}\label{eq:nointerabs}
    \mathbb P\Big(X\geq 4 \eta^{2k} n \Big) \leq 1/4\,.
\end{align}
Hence, there is a choice of~$\mathcal C$ for which all three events in \eqref{eq:sizeofAbs},~\eqref{eq:manyabs}, and~\eqref{eq:nointerabs} fail. 
That is, there is a subset~$\mathcal C\subseteq \mathscr A$ such that 
$|\mathcal C|\leq 8\eta^{4k/3} n$, 
$X\leq 4\eta^{2k}n$, 
and $|\mathcal C\cap \mathscr A_x|\geq 2\eta^{5k/3} n$ for every~$x\in V(G)$.
Consequently, after deleting at most~$X\leq 4\eta^{2k}n$ elements from~$\mathcal C$ we obtain a set~$\mathscr P\subseteq \mathscr A$ of pairwise vertex-disjoint~$(3k-1)$-paths such that
\begin{align}\label{eq:absorbingpaths}
    |\mathscr P|\leq 8\eta^{4k/3} n \quad\text{and}\quad |\mathscr P\cap \mathscr A_x|\geq 2\eta^{5k/3} n - 4\eta^{2k}n \geq \eta^{5k/3} n \text{ for every~$x\in V(G)$}\,.
\end{align}

We need to connect the $(3k-1)$-paths in~$\mathscr P$. 
To do that we modify them so that they are all good~$k$-paths.
This follows from a simple application of Lemma~\ref{lem:goodTk}.
More precisely, let~$P\in \mathscr P$ be a fixed~$(3k-1)$-path and let~$T_1$ and $T_2$ be the first and last~$3k-1$ vertices of~$P$, respectively.
As $P$ has~$6k-2$ vertices, $T_1$ and~$T_2$ are disjoint.
Further, let $T'_1$ be the first $2k-1$ vertices in $T_1$; let $T''_1$ be the last $k$ vertices in $T_1$; let $T'_2$ be the last $2k-1$ vertices in $T_2$; let $T''_2$ be the first $k$ vertices in $T_2$.
Since~$T'_1$ and~$T'_2$ induce transitive tournaments, Lemma~\ref{lem:goodTk} yields transitive tournaments on~$k$ vertices~$T^*_1\subseteq T'_1$ and~$T^*_2\subseteq T'_2$, where~$T^*_1$ is in-good and~$T^*_2$ is out-good.
Suppose~$P$ is a~$(3k-1)$-absorber for a vertex~$x\in V(G)$, meaning that both~$T_1T_2=T'_1T''_1T''_2T'_2$ and~$T_1xT_2$ form~$(3k-1)$-paths. 
It is easy to see that both~$T^*_1T''_1T''_2 T^*_2$ and~$T^*_1T''_1xT''_2 T^*_2$ induce $k$-paths in $G$ (with some additional edges), and hence~$P':=T^*_1T''_1T''_2 T^*_2$ is a stretched $k$-absorber for~$x$. 
In other words, after deleting~$2k-2$ vertices from~$P$ we obtain a good~$k$-path~$P'$ such that for every~$x\in V(G)$, if~$P$ is a~$(3k-1)$-absorber for~$x$, then~$P'$ is a stretched $k$-absorber for~$x$.

Repeat this argument for every~$P\in\mathscr P$ to obtain a new collection~$\mathscr P'$ of
vertex-disjoint
good $k$-paths and observe that due to \eqref{eq:absorbingpaths}, it holds that~$|\mathscr P'|\leq 8\eta^{4k/3}n$ and for every~$x\in V(G)$ there are at least~$\eta^{5k/3}n$ stretched $k$-absorbers in~$\mathscr P'$.

We connect the paths in~$\mathscr P'$ by greedily applying Lemma~\ref{lem:manyCon-oriented}. 
More precisely, let~$\mathscr P'=: \{P_1',\dots, P_r'\}$ where~$r\leq 8\eta^{4k/3}n$. 
Every $k$-path in~$\mathscr P'$ starts with an in-good transitive tournament and ends in an out-good transitive tournament, and has size $4k$.
Suppose that the $k$-paths~$P_1',\dots, P_i'$ are already connected using at most~$4k$ additional vertices for each connection. 
Then, the number of vertices used in those connections and in the $k$-paths in~$\mathscr P'$ is at most $(4k+4k)|\mathscr P'| \leq 8k \cdot 8\eta^{4k/3} n \leq \zeta n$. 
Hence, Lemma~\ref{lem:manyCon-oriented} yields a $k$-path connecting~$P_i'$ with~$P_{i+1}'$ using at most~$4k$ additional vertices and avoiding all vertices from previous connections and all vertices from~$\mathscr P'$. 
Let~$P_A$ be the resulting good~$k$-path and observe that
$$|V(P_A)|\leq (4k+4k) |\mathscr P'| \leq 64 k \eta^{4k/3}n \leq \eta^{k} n\,.$$ 
Hence~\ref{it:abs-size} follows. 
 
Finally, to see~\ref{it:abs-global}, let~$L\subseteq V(G)\setminus V(P_A)$ be a set of at most~$\eta^{2k} n$ vertices. 
For every~$x\in L$ there are~$\eta^{5k/3} n > \eta^{2k} n$ vertex-disjoint stretched $k$-absorbers for~$x$ in~$\mathscr P'$ (so they are subpaths of $P_A$). 
Therefore, we may greedily choose a distinct stretched $k$-absorber for each vertex~$x\in L$,
and use their absorbing properties to ensure we can indeed find the desired $k$-path as in (ii). 

\smallskip
\textbf{Choosing a reservoir.}
Choose a set~$S\subseteq V(G)\setminus V(P_A)$ of size~$\eta^{3k}n$ uniformly at random.
As $\delta^0(G) \ge \delta n + |V(P_A)|$,
for every vertex~$v\in V(G)$, we have
$\mathbb E(|N^+_G(v, S)|),  \mathbb E(|N^-_G(v, S)|) \geq \delta |S|$.
Similarly, 
for every out-good transitive tournament~$T^+$ in $G$, we have
$$
\mathbb E\bigg(\Big|\bigcap_{x\in V(T^+)} N^+_G(x, S)\Big|\bigg) \ge \eta^{3k} \Big|\bigcap_{x\in V(T^+)} N^+_G(x) \setminus V(P_A) \Big|\,
\ge \left(\frac{\left(2k-1\right)\delta - k + 1}{k2^{2k - 1}} - \eta^k \right)|S|
.$$
We say that an out-good transitive tournament~$T^+\subseteq V(G)\setminus S$ is \textit{out-bad for~$S$} if 
$$\left |\bigcap_{x\in V(T^+)} N^+_G(x, S)\right | 
< \frac{(2k-1)(\delta - \zeta^k)-k+1}{k2^{2k-1}} (|S|+k)\,;$$
that is,~$T^{+}$ is \textit{not}~$(\delta-\zeta^k)$-out-good in~$G[S\cup V(T^+)]$.
We define \textit{in-bad for $S$} analogously. 
Thus, using Chernoff's bound for the hypergeometric distribution together with the union bound, we obtain
\begin{align*}
    &\mathbb P\left(\delta^0(G[S])
    \leq \left(\delta - \frac{\zeta^k}{2}\right)|S|\right) 
    <1/2,\\
    &\mathbb P\left(\text{$\exists$ an out-good transitive tournament $T^{+}$ that is out-bad for~$S$}\right) 
    <1/4  \text{ and }
    \\
    &\mathbb P\left(\text{$\exists$ an in-good transitive tournament $T^{-}$ that is in-bad for~$S$}\right) 
    <1/4\,. 
\end{align*}
Hence, there is a set~$\mathfrak R\subseteq V(G)\setminus V(P_A)$ such that 
\begin{enumerate}[label={\rm (\alph*)}]
    \item \label{it:R-size} $|\mathfrak R| = \eta^{3k} n$ and
    \item \label{it:R-good} for any vertex-disjoint out-good transitive tournament~$T^+$ and in-good transitive tournament~$T^-$, both on $k$ vertices in~$V(G)\setminus (V(P_A)\cup \mathfrak R)$, we have 
    $$\delta^0(G[\mathfrak R\cup V(T^+)\cup V(T^-)] \geq \Big(\delta-\zeta^k\Big)|\mathfrak R|\,,$$
    and moreover in~$G[\mathfrak R\cup V(T^+)\cup V(T^-)]$ the tournaments~$T^{+}$ and~$T^-$ are~$(\delta-\zeta^k)$-out-good and~$(\delta-\zeta^k)$-in-good respectively.
\end{enumerate}

\smallskip
\textbf{Covering most vertices and final absorption.}
Our goal for this part is to find a~$k$-cycle that contains~$P_A$ as a subpath and covers all but at most~$\eta^{2k} n$ vertices from~$V(G)$. 
After that, we apply~\ref{it:abs-global} to finish the proof.

Let~$G':=G\setminus (V(P_A)\cup \mathfrak R)$ and $Q :=  10^{1000(2k-1)}$.
Apply Lemma~\ref{lem:reg} with parameters $\eps, d$ and $t_0 := 1/\eps^2$, to obtain a partition~$\{V_0, V_1, \dots, V_t\}$ of~$V(G')$ where~$1/\eps^2 \leq t \leq T$.
Thus, Lemma~\ref{prop:reducedOrientedDegree} yields a spanning oriented subgraph~$R_o$ of the reduced digraph of $G'$, with
\begin{align*}
\delta^0(R_o)
&\geq
\left( \frac{\delta^0(G')}n - 2 d  \right) t
\geq
\left( \frac{\delta^0(G) - |V(P_A)\cup \mathfrak R|} n - 2 d  \right) t\\
& \stackrel{\mathclap{\text{\ref{it:abs-size}, \ref{it:R-size}}}}{\geq}
\ \left( \frac{1}{2}-\frac{1}{c^k}-\eta^{3k} -\eta^{k} - 2 d  \right) t
\geq 
\left(\frac{1}{2}-\frac{2}{c^k}\right)t
 \ge \left(\frac{1}{2}-\frac{1}{(20Q)^3}\right)t \,. 
\end{align*}
Apply Lemma~\ref{lem:tiling} with~$R_o,2k-1$ playing the roles of $G,k$ to obtain a partition of $V(R_o)$ into~$(2k-1)$-cycles~$C_1,\dots, C_r$ each of them of size~$Q$ or~$Q+1$ and 
\begin{align*}
r\leq t/Q\,. 
\end{align*}

By our choice of~$\eps$, for every~$i\in[r]$, we can apply Lemma~\ref{lem:blowup} to~$C_i$ with~$\eta^{2k}/2$ playing the role of~$\eta$.
Thus, we obtain~$(2k-1)$-paths~$P_1, \dots, P_r$ covering all vertices of~$G'$ except at most
\begin{align}\label{eq:firstleftover}
  |V(G') \setminus \bigcup_{i \in [r]} V(P_i) | \le \eps n + \frac{\eta^{2k}}{2}\cdot \frac{n}{t}\cdot {t} \leq \frac{2\eta^{2k}n}{3}\,  
\end{align}
vertices.

Roughly speaking, we shall connect these $(2k-1)$-paths and the absorbing $k$-path~$P_A$ into a~$k$-cycle by applying Lemma~\ref{lem:goodTk} and property \ref{it:R-good} of~$\mathfrak R$ together with Lemma~\ref{lem:manyCon-oriented}.
In fact, we first modify these $(2k-1)$-paths $P_1, \dots, P_r$ to get good~$k$-paths.
For a fixed~$i\in [r]$, let~$T_1$ and~$T_2$ be the first and last~$(2k-1)$ vertices of~$P_i$, respectively. 
Due to Lemma~\ref{lem:goodTk}, there are $k$-vertex in-good and out-good transitive tournaments $T_1'\subseteq T_1$ and~$T_2'\subseteq T_2$, respectively. 
It is easy to see that if we delete from~$P_i$ the vertices in~$T_1\setminus T_1'$ and~$T_2\setminus T_2'$ then we obtain a~$k$-path from~$T_1'$ to~$T_2'$. 
Applying the same argument for every~$i\in [r]$ we obtain a collection of good~$k$-paths~$\{P_1',\dots, P_r'\}$ covering almost all vertices from~$G'$.
In particular, if~$L:= V(G')\setminus \big(\bigcup_{i\in [r]} V(P_i')\big)$ then, using~\eqref{eq:firstleftover} we have
\begin{align}\label{eq:leftover}
    |L|\leq  |V(G') \setminus \bigcup_{i \in [r]} V(P_i) | + 2(k-1)r \leq \frac{3\eta^{2k} n}{4} \,.
\end{align}

Set~$P_{r+1}':=P_A$. 
From now on operations over the indices are assumed to be modulo~$r+1$.
We now connect the~$k$-paths~$\mathcal P':=\{P_1',\dots, P_r', P_{r+1}'\}$ into the~$k$-cycle using only vertices from~$\mathfrak R$. 
Let $T_{j}^{+}$ and~$T^-_j$ be the out-good and in-good tournaments at the end of and the start of~$P_{j}'$, respectively. 
Suppose the $k$-paths~$P_1',\dots, P_{i-1}'$ are already connected into a single $k$-path $P_{i-1}^*$ using at most~$6k(i-1)$ additional vertices, all from~$\mathfrak R$. 
Furthermore, $P^*_{i-1}$ is obtained by connecting $T_{j-1}^{+}$ to $T_{j}^-$ for $j \in [i-1]$. 
Due to~\ref{it:R-good} and the fact that 
\[
\frac{1}{2} - \frac{1}{3^{18k}} < \frac{1}{2} - \frac{2}{c^k} - \zeta^k = \delta - \zeta^k , 
\]
we may apply Lemma~\ref{lem:manyCon-oriented} to the oriented graph~$G[\mathfrak R\cup V(T_{i-1}^+)\cup V(T_i^-)]$. 
More precisely, we have 
$$\delta^0(G[\mathfrak R\cup V(T_{i-1}^+)\cup V(T_i^-)]) \geq \Big(\delta-\zeta^k\Big)|\mathfrak R|\,,$$
and moreover in~$G[\mathfrak R\cup V(T_{i-1}^+)\cup V(T_i^-)]$ the tournaments~$T_{i-1}^{+}$ and~$T_i^-$ are~$(\delta-\zeta^k)$-out-good and~$(\delta-\zeta^k)$-in-good respectively.
Thus, since the number of vertices from $\mathfrak R \cup V(T_{i-1}^+)\cup V(T_i^-)$ used in connecting so far is at most~$6k(i-1)\leq 6k(r+1)\leq 7kt/Q \leq \zeta |\mathfrak R|$, Lemma~\ref{lem:manyCon-oriented} yields a $k$-path on at most~$6k$ vertices completely contained in~$\mathfrak R \cup V(T_{i-1}^+)\cup V(T_i^-)$, from~$T_{i-1}^+$ to $T^-_i$ and avoiding all previous connecting $k$-paths.
We obtain $P^*_i$ as desired.

Therefore, we can greedily connect all consecutive $k$-paths~$P_1',\dots, P_{r+1}'$ using  vertex-disjoint~$k$-paths completely contained in~$\mathfrak R$ (including connecting from  $P_{r+1}'$ to $P_1'$). 
Let~$H'$ be the~$k$-cycle obtained in this way. 
Note that~$H'$ covers all vertices from~$V(G)\setminus L$ except possibly some vertices in~$\mathfrak R$. 
Thus,
\begin{align*}
|V(G)\setminus V(H')| \leq |L| + |\mathfrak R| 
\overset{\text{\eqref{eq:leftover}, \ref{it:R-size}}}{\leq}
\frac{3\eta^{2k}}{4}n + \eta^{3k} n \leq \eta^{2k} n\,.
\end{align*}

Finally, using~\ref{it:abs-global}, we absorb all vertices in~$V(G)\setminus V(H')$ into $P_A$ (which is a subpath of $H'$) and obtain a~$k$th power of a Hamilton cycle in $G$.\qed

\subsection{Proof of Theorem~\ref{thm:square}}\label{sec:sketch}
Define additional constants $\gamma,  \xi>0$ and $n_0,T \in \mathbb N$ such that
\begin{align*}
0< 1/n_0\ll 1/T\ll \gamma \ll  \xi \ll  \eta  \,.
\end{align*} 
Given any $n \geq n_0$, let~$G$ be an $n$-vertex digraph as in the statement of the theorem. 

\smallskip
\textbf{Constructing the absorbing path.}
First, we find an absorbing $2$-path~$P_A$ in $G$ that starts with an edge $a_1a_2$ and ends with an 
edge $a_3a_4$ so that
\begin{enumerate}[label={\rm (\roman*)}]
    \item $|V(P_A)|\leq 32\gamma n$ and \label{it:abs-size-dir}
    \item for every set~$L\subseteq V(G)\setminus V(P_A)$ of size at most~$\gamma^{2}n$, 
    $G[L\cup V(P_A)]$ contains a spanning~$2$-path from~$a_1a_2$ to~$a_3a_4$. \label{it:abs-global-dir}
\end{enumerate}

To construct $P_A$, we proceed as in the proof of Theorem~\ref{thm:main} above. 
More precisely, given~$v\in V(G)$ define
$\mathscr A_v:=\{P \colon P \text{ is an absorber for~$v$}\}$ and 
~$\mathscr A:=\bigcup_{v\in V(G)}\mathscr A_v$. In particular, recalling Definition~\ref{defyabsy}, note that the sets in $\mathscr A$ each have size $8$.
Note that Lemma~\ref{lem:abs} with~$k=2$, implies that~$|\mathscr A_v|\geq \xi n^8$ for every~$v\in V(G)$.
Similarly to before, take a random subset~$\mathscr B\subseteq \mathscr A$, in which every element is taken with probability~$p:=\gamma n/|\mathscr A|$; note that standard concentration inequalities imply that
$$\mathbb P(|\mathscr B| \geq 2p|\mathscr A |) \leq 1/2
\quad \text{and}\quad
\mathbb P (\exists v\in V(G) \colon |\mathscr B\cap \mathscr A_v| \leq {p|\mathscr A_v|}/2 ) <1/4\,.$$
Moreover, if~$X$ is the random variable counting the number of pairs~$P,P'\in\mathscr B$ sharing at least one vertex, it is not hard to check that~$\mathbb P(X\geq \gamma^{3/2} n)< \frac{1}{4}$. 
Hence, we may pick a collection~$\mathscr B\subseteq \mathscr A$ for which all bad events above fail. 
In particular, after deleting at most~$\gamma^{3/2} n$ elements from such a collection, we obtain a collection $\mathscr P\subseteq \mathscr A$ of pairwise vertex-disjoint absorbers  satisfying
\begin{align}\label{usedused}
    |\mathscr P| \leq {2\gamma n}
\qquad \text{and} \qquad
|\mathscr P\cap \mathscr A_v|\geq \gamma^{3/2}n \text{ for every~$v\in V(G)$}.\,
\end{align}

 We can now sequentially connect up the absorbers in $\mathscr P$ into a single $2$-path $P_A$ via repeated applications of Lemma~\ref{lem:manyCon}. In each application of Lemma~\ref{lem:manyCon} we only introduce at most 16 new vertices; this ensures
 \ref{it:abs-size-dir} holds. Further, (\ref{usedused}) ensures \ref{it:abs-global-dir} holds.

\smallskip
\textbf{Choosing a reservoir.}
We choose a reservoir~$\mathfrak R\subseteq V(G)\setminus V(P_A)$, by taking a random set in which every vertex from  $V(G)\setminus V(P_A)$ is included independently with probability~$\gamma^2/4$.
With positive probability, we have
\begin{enumerate}[label={\rm (\alph*)}]
    \item \label{it:R-size-dir} $\gamma ^2n/8 \leq |\mathfrak R|\leq \gamma^2 n/2$ and
    \item \label{it:R-manyconnections-dir} given  any
     pair of vertex-disjoint edges~$ab, yz \in E(G)$,
     and  any set $U \subseteq \mathfrak R \setminus \{a,b,y,z\}$
    of size at most $\eta|\mathfrak R|/4 $,
     there is a  $2$-path $P$ on at most 20 vertices from~$ab$ to~$yz$ in $G[\mathfrak R \cup \{a,b,y,z\}]$ which avoids $U$.
\end{enumerate}
Indeed, one can ensure that  $\mathfrak R$ satisfies
$\delta (G[\mathfrak R \cup \{a,b,y,z\}])\geq (8/5+\eta /2)|\mathfrak R \cup \{a,b,y,z\}|$
for all choices of $ab, yz \in E(G)$; then \ref{it:R-manyconnections-dir} follows by applying 
Lemma~\ref{lem:manyCon}.

\smallskip
\textbf{Covering most vertices and final absorption.}
As before, we construct a $2$-cycle~$H'$ in~$G$ covering all vertices except at most~$\eta^2n$. 
Let~$G':=G\setminus (V(P_A)\cup \mathfrak R)$ and note that
$$\delta (G')\geq \Big(\frac{8}{5} + \eta - 32 \gamma - \frac{\gamma^2}{2} \Big)n \geq \Big(\frac{3}{2} + {\gamma^3}\Big)n\,.$$

Lemma~\ref{lem:almost} with~$k=2$, and $\gamma ^3$ playing the role of $\eta$,
implies the existence of at most~$T$ vertex-disjoint $2$-paths in $G'$, covering all but at most $\gamma^3 n$ vertices in $G'$. 
We can now greedily connect these $2$-paths and the absorbing $2$-path~$P_A$ via $\mathfrak R$ by iteratively applying property~\ref{it:R-manyconnections-dir}.  
Let~$H'$ be the resulting $2$-cycle.

We finish by noting that~$|V(G)\setminus V(H')|\leq \gamma^3 n + |\mathfrak R| \leq \gamma^2 n$. 
Thus, the absorbing property \ref{it:abs-global-dir} implies that~$H'$ can be extended to the square of a Hamilton cycle in $G$.\qed

\section{Further discussion and results}\label{sec:conc}
\subsection{Open problems and approaches to almost covering}\label{81}
In this paper we  asymptotically resolved Conjecture~\ref{conjnew} in the case of the square of a Hamilton cycle (Theorem~\ref{thm:square}). Obtaining a connecting lemma 
is the main
barrier to extending our proof to  $k$th powers of a Hamilton cycle more generally. For example, consider an $n$-vertex digraph $G$ with $\delta (G)= (5/3+o(1))n$. So certainly $G$ satisfies the minimum total degree condition from the $k=3$ case of Conjecture~\ref{conjnew}. However, $\delta (G)= (5/3+o(1))n$ only implies that $\delta ^+ (G)\geq (2/3+o(1))n$. Thus, there could be three vertices $x,y,z\in V(G)$ for which $|N^+_G(x) \cap N^+_G(y) \cap N^+_G(z)|=o(n)$.

Even worse, if one considers a digraph $G$ as in the $k \geq 4$ cases of Conjecture~\ref{conjnew}, then there can be $k$-sets of vertices in $G$ without a single common out-neighbor. In this case, it is impossible to prove a direct analog of Lemma~\ref{lem:manyCon} where one can connect any $k$-set of vertices to another $k$-set via a short $k$-path.
Instead, one would likely need a connecting lemma more akin to Lemma~\ref{lem:manyCon-oriented}. 

\smallskip

It is also natural to look at minimum degree conditions that force a $k$th power of a Hamilton \emph{path} in a digraph or oriented graph. 
For example, Ghouila-Houri~\cite{GhouilaHouri} proved  
that every  $n$-vertex digraph $G$ with minimum total degree $\delta(G)\geq n-1$ contains a  Hamilton path.\footnote{Note  one does not require that $G$ is strongly connected here.}
 Notice that the extremal construction for Conjecture~\ref{conjnew} given in Proposition~\ref{propextremal} actually contains the $k$th power of a Hamilton {path}.
We suspect that the minimum total degree threshold for forcing a 
$k$th power of a Hamilton {path} in a digraph is lower than the corresponding threshold in Conjecture~\ref{conjnew}. Such a phenomenon may also occur for   the minimum semi-degree problem in oriented graphs. Note that the former problem is closely linked to the problem of relaxing the minimum total degree condition in Conjecture~\ref{conjnew} at the expense of a connectivity-type condition.

\smallskip

The almost covering lemma (Lemma~\ref{lem:tiling}) for Theorem~\ref{thm:main} has a different flavour to analogous results in the literature. Indeed, 
a standard approach is to follow the strategy that was used in~\cite{kss} to prove Conjecture~\ref{seymour} for large graphs. In~\cite{kss}, the authors  find a  $K_{k+1}$-tiling in the reduced graph $R$ of the host graph $G$, so that most vertices in $R$ are covered by these cliques.
Then they use an (undirected) graph analog of Lemma~\ref{lem:blowup} to `wind around' these copies of $K_{k+1}$ in $R$ to obtain a collection of vertex-disjoint $k$-paths in $G$ covering most of the vertices.

Thus, a natural approach to Theorem~\ref{thm:main} would have been to first obtain a minimum semi-degree condition that forces an almost spanning $C^k_{\ell}$-tiling in an oriented graph $G$ where $\ell \geq 2k+1 $. Then one would be able to make use of Lemma~\ref{lem:blowup} to obtain an almost covering lemma.

The difficulty with this approach is that it is not immediately clear that one can even obtain such a $C^k_{\ell}$-tiling result. For example, a special case of a result of Bollob{\'a}s and H{\"a}ggkvist \cite[Theorem 4]{BH}
tells us that
there are arbitrarily large regular tournaments (i.e., oriented graphs with the largest possible minimum semi-degree) that do not even contain a single copy of $C^2_{5}$. This highlights that both   Tur\'an-type  and  $H$-tiling problems in oriented graphs are much more subtle than their (undirected) graph analogs. We discuss these problems further in the next subsection.

\subsection{Tur\'an-type and tiling problems for oriented graphs}\label{sec:turan}
We say that an oriented graph $H$ is \emph{Tur\'anable} if there exists $n_0 \in \mathbb N$ such that for all oriented graphs~$G$ on $n\geq n_0$ vertices, if $\delta^0(G)\geq \floor{(n-1)/2}$, then $H\subseteq G$.
Thus, in the previous subsection we stated that $C^2_{5}$ is not Tur\'anable.
If we restrict to odd~$n$, then $H$ being Tur\'anable is equivalent to $H$ being contained in every sufficiently large regular tournament.  This case  corresponds to the notion of \emph{omnipresent}, which  was introduced by  Bollob{\'a}s and H{\"a}ggkvist~\cite{BH} more than three decades ago. 

Let $D_r$ denote the tournament on $3r$ vertices obtained from the $r$-blow-up $C_3(r)$ of the directed cycle $C_3$  by replacing the three  independent sets in $C_3(r)$ with transitive tournaments.  The next theorem is a slight generalization of one of Bollob{\'a}s and H{\"a}ggkvist's results, rephrased into our language.\footnote{While the proof of Theorem~\ref{thm:BH} is essentially identical to the proof of Theorem~4 in~\cite{BH}, we note that the statement of the latter theorem does not quite imply the former (unless one can prove that omnipresent implies Tur\'anable).  This is because in order for a graph $H$ to be Tur\'anable, it must also hold that for every sufficiently large even $n$, $H$ is a subgraph of every oriented graph $G$ on $n$ vertices with minimum semi-degree at least $\frac{n}{2}-1$, which may not even be a tournament.  The forwards direction of Theorem 8.1 follows from Theorem~4 in~\cite{BH} because Tur\'anable clearly implies omnipresent.  The upcoming Proposition~\ref{prop:Dr} provides the proof for the backwards direction of Theorem 8.1.}

\begin{thm}[Bollob{\'a}s and H{\"a}ggkvist \cite{BH}]\label{thm:BH}
A tournament $T$ is Tur\'anable if and only if $T\subseteq D_r$ for some $r \in \mathbb N$.
\end{thm}

Given~$\ell\geq 2k+1$, it is not hard to see that~$C^k_\ell \subseteq D_r$ for some~$r\in \mathbb N$ if and only if~$\ell\geq 3k$. 
In the next subsection, we prove that Theorem~\ref{thm:BH} also holds for power of cycles.

\begin{theorem}\label{thm:turancycles}
Let $k\geq 1$ and $\ell\geq 2k+1$.  $C^k_\ell$ is Tur\'anable if and only if $C^k_\ell \subseteq D_r$ for some~$r\in \mathbb N$.
Equivalently, $C^k_\ell$ is Tur\'anable if and only if $\ell\geq 3k$.
\end{theorem}

In fact, we believe Theorem~\ref{thm:BH} can be extended to all oriented graphs.

\begin{conjecture}\label{con:turan}
An oriented graph~$H$ is Tur\'anable if and only if there is an~$r\in\NN$ with~$H\subseteq D_r$.  
\end{conjecture}

It is also natural to look at the analogous problem for $H$-factors. We say that an oriented graph~$H$ is \emph{tileable} if there exists $n_0 \in \mathbb N$ such that for all oriented graphs $G$ on $n\geq n_0$ vertices with $n$ divisible by $|H|$, if $\delta^0(G)\geq \floor{(n-1)/2}$, then $G$ contains an $H$-factor.
For example, it is known that all acyclic oriented graphs (including
transitive tournaments)
are tileable; see~\cite[Theorem~1.3]{Yuster}. Directed cycles are also tileable; see~\cite{MR4025428,MR4763245}.

The next question is perhaps the most natural starting point in the study of tileability.
\begin{question}\label{quest:turantile}
Let $H$ be an oriented graph. Is it true that $H$ is tileable if and only if $H$ is Tur\'anable?
\end{question}

Once one knows that an oriented graph $F$ is Tur\'anable, the next step is to determine the minimum semi-degree threshold for forcing a copy of $F$ in an oriented graph. Thus,
given~$n\in \mathbb N$ and a Tur\'anable oriented graph~$F$, define 
$$\text{ex}^0(n,F):= \max\{\delta^0(G)\colon G \text{ is an~$n$-vertex oriented graph with~$F\not\subseteq G$}\}\,.$$
Further, set $\kappa^0(F) := \lim_{n\to \infty} \tfrac{\text{ex}^0(n,F)}{n}\in [0,1/2]$.  The proof that this limit exists can be shown similarly as in \cite[Proposition 1.2]{MZ}.
In Section~\ref{sec:Turanbounds} we prove the following result.

\begin{proposition}\label{prop:Dr}
 There exists a constant $K>1$ such that, for each $r\in \mathbb N$, $\kappa^0(D_r) < \tfrac{1}{2}-\tfrac{1}{K^r}$.
\end{proposition}
Note that Proposition~\ref{prop:Dr} is quite a useful result.
Indeed, given any oriented graph $F$ such that $F \subseteq D_r$, we have that $\kappa^0(F) \leq \kappa^0(D_r)$; so Proposition~\ref{prop:Dr} provides an upper bound on $\kappa^0(F)$.

In fact, if Conjecture~\ref{con:turan} holds, then~$\kappa^0(F)<1/2$ for every Tur\'anable oriented graph~$F$. 
Moreover, since~$C^k_\ell\subseteq D_\ell$ for every~$\ell\geq 3k$ we have the following corollary for powers of cycles. 

\begin{corollary}\label{cor:Turancycles}%
There is a~$K> 1$ such that for every~$k\in \mathbb N$ and~$\ell\geq 3k$ we have~$\kappa^0(C^k_{\ell})\leq \tfrac{1}{2}-\tfrac{1}{K^\ell}$.
\end{corollary}

Note that the problem of determining 
$\kappa^0(C_{3})$ is a special case of the minimum semi-degree version of the Caccetta--H\"aggkvist conjecture, and it is known that $1/3\leq \kappa^0(C_{3})\leq 0.343545$; see~\cite{Lich}.

For $k\geq 2$, the best lower bound construction we found for an  oriented graph with high minimum semi-degree not containing a copy of~$C^k _{3k}$ is given by the blow-up of a semi-regular tournament on $3k-1$ vertices.
It would be interesting to determine if this is indeed the best construction, at least for the first non-trivial case.

\begin{question}\label{quest:osT}
    Is it true that~$\kappa^0(C^2_{6}) = 2/5$?
\end{question}

\begin{remark}
After this paper first appeared online, Araujo and Xiang \cite{AX} made progress on the above-mentioned problems.  They disproved Conjecture \ref{con:turan}, and answered Questions \ref{quest:turantile} and \ref{quest:osT} in the negative.  In particular, regarding Question \ref{quest:turantile}, they showed that for all $r\geq 2$, $D_r$ is not tileable.  
\end{remark}

\subsection{Proof of Theorem \ref{thm:turancycles}}

Define the tournament~$F_r$ iteratively as follows. 
Let~$F_1:=C_3$.
For~$r\geq 2$, let~$F_r$ be the vertex-disjoint union of three copies of~$F_{r-1}$, namely $F^1$, $F^2$, and~$F^3$, and add all edges from ~$V(F^1)$ to~$V(F^2)$, from~$V(F^2)$ to~$V(F^3)$, and from~$V(F^3)$ to~$V(F^1)$. 
Note that $F_r$ is a regular tournament on $3^r$ vertices.

Before proving Theorem~\ref{thm:turancycles}, we make one observation regarding the tournaments $F_r$.  

\begin{observation}\label{obs:Fk}Let $k\in \mathbb N$.  
    If $2k+1\leq \ell\leq 3k-1$, then~$C_\ell^k\not\subseteq F_r$ for all $r\in \mathbb N$.
\end{observation}
\begin{proof}
    Let~$\ell\geq 2k+1$ and $r\in \mathbb N$. 
    Observe that if~$|F_r|=3^r<\ell$ then clearly ${C}_{\ell}^k\not\subseteq F_r$.
 
    Let $C:={C}_{\ell}^k$ be a $k$-cycle with vertices~$v_1,\dots, v_\ell$ and suppose   that~$C\subseteq F_r$ (and consequently $3^r\geq \ell$).  We shall prove by induction on $r$ that~$\ell\geq 3k$. 

If $r=1$ then $C\subseteq F_1$ implies that $C=C_3$, i.e., $k=1$ and $\ell =3$; so indeed $\ell\geq 3k$.
For the inductive step,
    let~$V_1$,~$V_2$, and~$V_3$ be the three vertex classes such that~$V_i:=V(F^i)$ for~$i\in [3]$, where the $F^i$ correspond to the three copies of~$F_{r-1}$ given in the definition of~$F_r$.
    If~$V(C)\subseteq V_i$ for some~$i\in [3]$, then~$C\subseteq F_{r-1}$ and we conclude by induction.
    Thus, we may assume that this is not the case; notice that this actually implies that $C$ intersects each of $V_1$, $V_2$, and $V_3$.
    Therefore, 
    without loss of generality, suppose~$V_1$ is the class with the smallest number of vertices from $C$ and that~$v_1\in V_3$ and~$v_2\in V_1$.
    Then, the~$k-1$ successors of~$v_{2}$ in $C$, namely $v_{3}, \dots, v_{k+1}$, are all contained in~$V_1$, since they are all in the out-neighborhood of both~$v_{1}$ and~$v_{2}$. 
    In particular,~$|V(C)\cap V_1|\geq k$. 
    Since~$V_1$ is the class with the smallest amount of vertices from~$C$, then~$\ell\geq 3k$.
\end{proof}








Now Theorem \ref{thm:turancycles} is an immediate consequence of Observation \ref{obs:Fk} and Theorem \ref{thm:BH}. 

\begin{proof}[Proof of Theorem \ref{thm:turancycles}]
If $\ell\geq 3k$, then ${C}_{\ell}^k\subseteq D_\ell$ and since $D_\ell$ is Tur\'anable by Theorem \ref{thm:BH} we have that  ${C}_{\ell}^k$ is Tur\'anable.

If $2k+1\leq \ell\leq 3k-1$, then Observation~\ref{obs:Fk} implies~${C}_\ell^k\not\subseteq F_r$ for all $r\in \mathbb N$.  Since $F_r$ is itself a regular tournament on $3^r$ vertices with $\delta^0(F_r)=\frac{1}{2}(3^r-1)$, $C_\ell^k$ is not Tur\'anable by definition.
\end{proof}

\subsection{Proof of Proposition~\ref{prop:Dr}}
\label{sec:Turanbounds}
The proof  of Proposition~\ref{prop:Dr} makes use of
 Lemma~\ref{lem:tiling}.
 We will also use the following result which  just follows immediately by combining Lemma~1.3 and Theorem~3.5 from~\cite{FS1}. 

\begin{theorem}[Fox and Sudakov \cite{FS1}]\label{thm:FS}
Given any $\eps >0$, there exists 
$\delta>0$ such that the following holds for all $r \in \mathbb N$
and $n\geq \delta^{-4r/\delta}$.  
If $T$ is an $n$-vertex tournament with $\delta ^0 (T) \geq \eps n$, then $D_r\subseteq T$.
\end{theorem}






Now we are ready to prove Proposition~\ref{prop:Dr}.

\begin{proof}[Proof of Proposition~\ref{prop:Dr}]
Let~$\delta>0$ be the output of  Theorem~\ref{thm:FS} on input~$\eps:=2/5$.
Let $K:= \max \{ 20^{3}\cdot 10^{6000}, 20 ^3 \cdot \delta^{-12/\delta} \}$ and
$Q:= \max \{10^{2000},\delta^{-4r/\delta} \}$; so $K^r \ge (20Q)^3$.

Suppose $G$ is a sufficiently large $n$-vertex oriented graph
with $\delta ^0 (G) \geq (1/2-1/K^r)n$. By Lemma~\ref{lem:tiling} (with $k=2$), 
certainly there exists a tournament $T$
 on~$Q$ or~$Q+1$ vertices in $G$,
 such that $\delta ^0 (T) \geq 2|T|/5$.
 As $|T|\geq \delta^{-4r/\delta}$, Theorem~\ref{thm:FS} implies that 
 $D_r \subseteq T \subseteq G$, as desired.
\end{proof}

\section*{Acknowledgment}
The authors are  grateful to the referee for their careful review.

\smallskip

{\noindent \bf Open access statement.}
	This research was funded in part by  EPSRC grants  EP/V002279/1 and EP/V048287/1. For the purpose of open access, a CC BY public copyright licence is applied to any Author Accepted Manuscript arising from this submission.

{\noindent \bf Data availability statement.}
There are no additional data beyond that contained within the main manuscript.

\bibliographystyle{abbrv}
\bibliography{references.bib}

\begin{thebibliography}{10}

\bibitem{ABHKP}
P.~Allen, J.~B{\"o}ttcher, H.~H{\`a}n, Y.~Kohayakawa, and Y.~Person.
\newblock Blow-up lemmas for sparse graphs.
\newblock {\em Disc. Anal.}, 8:141pp, 2025.

\bibitem{AlonShapira04}
N.~Alon and A.~Shapira.
\newblock Testing subgraphs in directed graphs.
\newblock {\em J. Comput. System Sci.}, 69(3):353--382, 2004.

\bibitem{AX}
I.~Araujo and Z.~Xiang.
\newblock On the {T}ur\'anability and tileability of oriented graphs.
\newblock {\em arXiv preprint arXiv:2507.13267}, 2025.

\bibitem{BH}
B.~Bollob{\'a}s and R.~H{\"a}ggkvist.
\newblock Powers of {H}amilton cycles in tournaments.
\newblock {\em J. Combin. Theory Ser. B}, 50(2):309--318, 1990.

\bibitem{bandwidth}
J.~B{\"o}ttcher, M.~Schacht, and A.~Taraz.
\newblock Proof of the bandwidth conjecture of {B}ollob{\'a}s and {K}oml{\'o}s.
\newblock {\em Math. Ann.}, 343(1):175--205, 2009.

\bibitem{BS}
M.~Buci{\'c} and B.~Sudakov.
\newblock Tight {R}amsey bounds for multiple copies of a graph.
\newblock {\em Adv. Combin.}, 2023:22pp, {F}ebruary 2023.

\bibitem{cdk}
P.~Ch{\^a}u, L.~DeBiasio, and H.~A. Kierstead.
\newblock P{\'o}sa's conjecture for graphs of order at least $2\times 10^8$.
\newblock {\em Random Structures \& Algorithms}, 39(4):507--525, 2011.

\bibitem{cliquetilings}
A.~Czygrinow, L.~DeBiasio, T.~Molla, and A.~Treglown.
\newblock Tiling directed graphs with tournaments.
\newblock {\em Forum Math. Sigma}, 6:e2, 2018.

\bibitem{etal}
N.~Dragani{\'c}, F.~Dross, J.~Fox, A.~Gir{\~a}o, F.~Havet, D.~Kor{\'a}ndi,
  W.~Lochet, D.~M. Correia, A.~Scott, and B.~Sudakov.
\newblock Powers of paths in tournaments.
\newblock {\em Combin. Probab. Comput.}, 30(6):894--898, 2021.

\bibitem{DMS}
N.~Dragani{\'c}, D.~Munh{\'a}~Correia, and B.~Sudakov.
\newblock Tight bounds for powers of {H}amilton cycles in tournaments.
\newblock {\em J. Combin. Theory Ser. B}, 158:305--340, 2023.

\bibitem{eb}
O.~Ebsen, G.~S. Maesaka, C.~Reiher, M.~Schacht, and B.~Sch{\"u}lke.
\newblock Embedding spanning subgraphs in uniformly dense and inseparable
  graphs.
\newblock {\em Random Structures \& Algorithms}, 57(4):1077--1096, 2020.

\bibitem{posa}
P.~Erd\H{o}s.
\newblock Problem 9, {T}heory of {G}raphs and its {A}pplications ({M}.
  {F}ieldler ed.).
\newblock {\em Czech. Acad. Sci. Publ., Prague}, page 159, 1964.

\bibitem{EM}
P.~Erd{\H{o}}s and L.~Moser.
\newblock On the representation of directed graphs as unions of orderings.
\newblock {\em Publ. Math. Inst. Hungar. Acad. Sci.}, 9(1-2):125--132, 1964.

\bibitem{fan0}
G.~Fan and R.~H{\"a}ggkvist.
\newblock The square of a {H}amiltonian cycle.
\newblock {\em SIAM J. Discr. Math.}, 7(2):203--212, 1994.

\bibitem{fan1}
G.~Fan and H.~A. Kierstead.
\newblock The square of paths and cycles.
\newblock {\em J. Combin. Theory Ser. B}, 63(1):55--64, 1995.

\bibitem{fan2}
G.~Fan and H.~A. Kierstead.
\newblock Hamiltonian square-paths.
\newblock {\em J. Combin. Theory Ser. B}, 67(2):167--182, 1996.

\bibitem{fan3}
G.~Fan and H.~A. Kierstead.
\newblock Partitioning a graph into two square-cycles.
\newblock {\em J. Graph Theory}, 23(3):241--256, 1996.

\bibitem{fau}
R.~Faudree, R.~J. Gould, M.~S. Jacobson, and R.~H. Schelp.
\newblock On a problem of {P}aul {S}eymour.
\newblock {\em Recent Advances in Graph Theory (VR Kulli ed.), Vishwa
  International Publication}, 197215, 1991.

\bibitem{FS1}
J.~Fox and B.~Sudakov.
\newblock Unavoidable patterns.
\newblock {\em J. Combin. Theory Ser. A}, 115(8):1561--1569, 2008.

\bibitem{FS2}
J.~Fox and B.~Sudakov.
\newblock Dependent random choice.
\newblock {\em Random Structures \& Algorithms}, 38(1-2):68--99, 2011.

\bibitem{GhouilaHouri}
A.~Ghouila-Houri.
\newblock Une condition suffisante d'existence d'un circuit hamiltonien.
\newblock {\em C.R.~Acad.~Sci.~Paris}, 251(4):495--497, 1960.

\bibitem{kkolms}
P.~Keevash, D.~K{\"u}hn, and D.~Osthus.
\newblock An exact minimum degree condition for {H}amilton cycles in oriented
  graphs.
\newblock {\em J. London Math. Soc.}, 79(1):144--166, 2009.

\bibitem{KKO-diracoriented}
L.~Kelly, D.~K{\"u}hn, and D.~Osthus.
\newblock A {D}irac-type result on {H}amilton cycles in oriented graphs.
\newblock {\em Combin. Probab. Comput.}, 17(5):689--709, 2008.

\bibitem{kssposa}
J.~Koml{\'o}s, G.~N. S{\'a}rk{\"o}zy, and E.~Szemer{\'e}di.
\newblock On the square of a {H}amiltonian cycle in dense graphs.
\newblock {\em Random Structures \& Algorithms}, 9(1-2):193--211, 1996.

\bibitem{blow}
J.~Koml{\'o}s, G.~N. S\'ark\"ozy, and E.~Szemer\'edi.
\newblock {B}low-up lemma.
\newblock {\em Combinatorica}, 17:109--123, 1997.

\bibitem{kssjgt}
J.~Koml{\'o}s, G.~N. S{\'a}rk{\"o}zy, and E.~Szemer{\'e}di.
\newblock On the {P}{\'o}sa--{S}eymour conjecture.
\newblock {\em J. Graph Theory}, 29(3):167--176, 1998.

\bibitem{kss}
J.~Koml{\'o}s, G.~N. S{\'a}rk{\"o}zy, and E.~Szemer{\'e}di.
\newblock Proof of the {S}eymour conjecture for large graphs.
\newblock {\em Annals of Combinatorics}, 2:43--60, 1998.

\bibitem{lev}
I.~Levitt, G.~N. S{\'a}rk{\"o}zy, and E.~Szemer{\'e}di.
\newblock How to avoid using the regularity lemma: P{\'o}sa’s conjecture
  revisited.
\newblock {\em Discr. Math.}, 310(3):630--641, 2010.

\bibitem{MR4025428}
L.~Li and T.~Molla.
\newblock Cyclic triangle factors in regular tournaments.
\newblock {\em Electron. J. Combin.}, 26(4):Paper No. 4.24, 23, 2019.

\bibitem{Lich}
N.~{L}ichiardopol.
\newblock A new bound for a particular case of the {C}accetta--{H}{\"a}ggkvist
  conjecture.
\newblock {\em Discr. {M}ath.}, 310(23):3368--3372, 2010.

\bibitem{MZ}
D.~Mubayi and Y.~Zhao.
\newblock Co-degree density of hypergraphs.
\newblock {\em Journal of Combinatorial Theory, Series A}, 114(6):1118--1132,
  2007.

\bibitem{RRS}
V.~R{\"o}dl, A.~Ruci{\'n}ski, and E.~Szemer{\'e}di.
\newblock A {D}irac-type theorem for 3-uniform hypergraphs.
\newblock {\em Combin. Probab. Comput.}, 15(1-2):229--251, 2006.

\bibitem{seymour}
P.~Seymour.
\newblock Problem section.
\newblock In {\em Combinatorics: Proceedings of the British Combinatorial
  Conference}, volume 1974, pages 201--202, 1973.

\bibitem{staden}
K.~Staden and A.~Treglown.
\newblock The bandwidth theorem for locally dense graphs.
\newblock {\em Forum Math. Sigma}, 8:e42, 2020.

\bibitem{treg}
A.~Treglown.
\newblock A note on some embedding problems for oriented graphs.
\newblock {\em J. Graph Theory}, 69(3):330--336, 2012.

\bibitem{MR4763245}
Z.~Wang, J.~Yan, and J.~Zhang.
\newblock Cycle-factors in oriented graphs.
\newblock {\em J. Graph Theory}, 106(4):947--975, 2024.

\bibitem{Yuster}
R.~Yuster.
\newblock Tiling transitive tournaments and their blow-ups.
\newblock {\em Order}, 20:121--133, 2003.

\end{thebibliography}

\end{document}